\documentclass[10pt]{amsart}
\usepackage{latexsym, amsmath,amssymb}
\usepackage{graphicx}
\usepackage{hyperref}
\usepackage{enumerate}
\usepackage{color}

\renewcommand{\epsilon}{\varepsilon}
\newtheorem{theorem}{Theorem}[section]
\newtheorem{lemma}[theorem]{Lemma}

\newtheorem{corr}[theorem]{Corollary}

\newtheorem{proposition}[theorem]{Proposition}
\newtheorem{deff}[theorem]{Definition}
\newtheorem{remark}[theorem]{Remark}
\newcommand{\bth}{\begin{theorem}}
\newcommand{\ble}{\begin{lemma}}
\newcommand{\bcor}{\begin{corr}}

\newcommand{\bdeff}{\begin{deff}}

\newcommand{\bprop}{\begin{proposition}}
\newcommand{\ele}{\end{lemma}}
\newcommand{\ecor}{\end{corr}}
\newcommand{\edeff}{\end{deff}}

\newcommand{\eprop}{\end{proposition}}

\newcommand{\la}{\lambda}

\newcommand{\eps}{\varepsilon}

\newcommand{\supp}{\text{supp }}
\renewcommand{\Pi}{\varPi}

\renewcommand{\epsilon}{\varepsilon}

\newcommand{\R}{{\mathbb R}}
\newcommand{\ls}{\lesssim}
\newcommand{\gs}{\gtrsim}
\newcommand{\1}{{\rm 1\hspace*{-0.4ex}%
		\rule{0.1ex}{1.52ex}\hspace*{0.2ex}}}

\numberwithin{equation}{section}

\begin{document}
		\title[]{Restriction of Schr\"odinger eigenfunctions to submanifolds}
		\author{Xiaoqi Huang, Xing Wang and Cheng Zhang}
			\address{Department of Mathematics\\
			Louisiana State University\\
			Baton Rouge, LA 70803, USA}
		\email{xhuang49@lsu.edu}
		\address{School of Mathematics\\
			Hunan University\\
			Changsha, HN 410012, China}
		\email{xingwang@hnu.edu.cn}
		\address{Mathematical Sciences Center\\
			Tsinghua University\\
			Beijing, BJ 100084, China}
		\email{czhang98@tsinghua.edu.cn}
		\date{}
		\begin{abstract}For Schr\"odinger operators $H_V=-\Delta_g+V$ with critically singular potentials $V$ on compact manifolds, we prove sharp estimates for the restriction of   eigenfunctions to submanifolds. Our method refines the perturbative argument by Blair-Sire-Sogge \cite{BSS} and enables us to deal with submanifolds of all codimensions. As applications, we obtain improved estimates on negatively curved manifolds and flat tori. In particular, we extend the uniform $L^2$ restriction estimates on flat tori by Bourgain-Rudnick \cite{br2012} to singular potentials.
		
		\end{abstract}
		
		\maketitle
		\section{Introduction}
Let $n\ge2$. Let $(M,g)$ be a compact, smooth, $n$-dimensional Riemannian manifold without boundary, and let $\Delta_g$ denote the associated Laplace-Beltrami operator. The concentration properties of the eigenfunctions of $\Delta_g$ depend on the geometry of $(M,g)$. On manifolds with integrable geodesic flows, such as the sphere, eigenfunctions like zonal harmonics or Gaussian beams may be highly concentrated near certain submanifolds \cite{Sogge86}. In contrast, on manifolds with chaotic geodesic flows, such as hyperbolic surfaces, high-energy eigenfunctions are usually expected to become equidistributed \cite{RS94,DJ18,DJN22,z1987}. Nevertheless, quantum scarring may still occur near unstable closed orbits in some chaotic systems \cite{Heller1984}. When studying quantum systems associated with Schr\"odinger operators $H_V=-\Delta_g+V$, the presence of a potential $V$ complicates the behavior of these systems. A fundamental principle is that eigenfunctions are confined by potential barriers, decaying exponentially in classically forbidden regions where the potential energy exceeds the particle's energy level. There are also other types of localization phenomena caused by potentials. For example, Anderson localization \cite{Anderson}, which occurs for certain disordered potentials, is due to destructive interference from scattering. The locations of these localized states have recently been shown to be predictable by landscape functions \cite{FM2012, ADFJM}. Notably,  potentials  with singularities on submanifolds may alter the chaotic nature of the Hamiltonian flow in the potential-free case on hyperbolic manifolds.  In this paper, we aim to quantitatively investigate the concentration behaviors of Schr\"odinger eigenfunctions with  singular potentials.

To measure the concentration of an eigenfunction, an important approach is to study the growth of its $L^p$ norms restricted to submanifolds. This type of estimates was studied by Reznikov \cite{rez04} for Maass forms on hyperbolic surfaces and by Burq--G\'erard--Tzvetkov \cite{BGT},  Hu \cite{Hu} for Laplacian eigenfunctions on general compact manifolds. To be specific, let  $\Sigma\subset M$ be a smoothly embedded submanifold of dimension $k$. For $2\le p\le \infty$, let
\begin{equation}\label{deltavalue}
	\delta(k,p)=\begin{cases}
		\dfrac{n-1}{2}-\dfrac{k}{p}, \quad\quad\quad\quad\text{if}~1\leq k\leq n-2,\ \ 2\leq p\leq\infty,\\
		\dfrac{n-1}{4}-\dfrac{n-2}{2p},  \quad\quad \text{if}~k=n-1,\  2\leq p \le \dfrac{2n}{n-1},\\
		\dfrac{n-1}{2}-\dfrac{n-1}{p}, \quad\quad\text{if}~k=n-1,\  \dfrac{2n}{n-1}< p\leq\infty.
	\end{cases}
\end{equation}
For the Laplacian eigenfunction $e_\la$ with $-\Delta_g e_\la=\la^2e_\la$, Burq--G\'erard--Tzvetkov \cite{BGT} and Hu \cite{Hu} proved the following sharp restriction bounds
\begin{equation}\label{bgt}
	\|e_\lambda\|_{L^p(\Sigma)}\ls\lambda^{{\delta(k,p)}}\|e_\lambda\|_{L^2(M)},
\end{equation}
albeit with a potential $(\log\lambda)^{\frac{1}{2}}$ loss when $(k,p)=(n-2,2)$. The log loss  can be removed under certain geometrical assumptions but remains generally open, see Chen--Sogge \cite{chensogge}, Wang--Zhang \cite{wzadv} for detailed discussions. Related works on restrictions of eigenfunctions also include those by Greenleaf-Seeger \cite{gs1994}, Tataru \cite{tataru}, Bourgain \cite{bo2009}, Tacy \cite{tacy2010}, Bourgain--Rudnick \cite{br2012}, Chen \cite{chen}, Xi--Zhang \cite{xz}, Hezari \cite{hez}, Blair \cite{mb}, Zhang \cite{zhang}, Huang--Zhang \cite{hzapde}, and Park \cite{park}. These types of restriction estimates are related to the stabilization of weakly damped wave equations, as discussed in Burq--Zuily \cite{BZ2016}.

Next, we recall some important classes of potentials in  spectral theory. To ensure that $H_V$ is essentially self-adjoint and bounded from below and its eigenfunctions are bounded, the ``minimal condition'' for $V$ is that it belongs to the Kato class, $\mathcal{K}$. The definition will be given in the next section. The spaces $L^{n/2}$ and $\mathcal{K}$ have the same scaling properties, and both obey the scaling law of the Laplacian, which accounts for their criticality. On the compact manifold $M$, we have  $L^q(M)\subset \mathcal{K}(M)\subset L^1(M)$  for $q>n/2$. The spectrum of $H_V$ for  $V\in \mathcal{K}(M)$ is discrete, and the eigenfunctions  are continuous, which allows for their restriction to submanifolds.  For a detailed introduction to Kato potentials and their physical motivations, see Simon \cite{SimonSurvey} and Blair-Sire-Sogge \cite{BSS}. 

\subsection{Main results}

We shall assume throughout that $M$ is a compact manifold of dimension $n\ge2$, and  $\Sigma$ be a submanifold of dimension $k$, and the potentials $V$ are real-valued. Let $\la\ge10$ and let $e_\la$ be an eigenfunction of $H_V$ with $$H_V e_\la=\la^2e_\la.$$ Let $2\le p\le \infty$ and let $\Lambda = \Lambda(k,p)$ denote the classical bounds  in \eqref{bgt}, i.e.
$$\Lambda(k,p) = \begin{cases}
	\lambda^{\delta(k,p)},   ~~~~\quad\quad\quad\text{ if}~(k,p)\neq(n-2,2) ,\\
	\lambda^{\delta(k,p)}(\log\lambda)^\frac{1}{2},~~~~\text{if}~(k,p)=(n-2,2).
\end{cases}$$
First, we prove essentially sharp estimates for the restriction of Schr\"odinger eigenfunctions to submanifolds of all codimensions.
\begin{theorem}\label{mainthm}
 Let $V\in L^{n/2}(M)\cap \mathcal{K}(M)$.  Then for all $n,k,p$ we have
	\begin{equation}\label{SPVbd}
		\|e_\la\|_{ L^p(\Sigma)}\ls \Lambda\|e_\la\|_{L^2(M)},
	\end{equation}
albeit with a potential logarithmic loss  if $(n,k,p)\in  \{(n,k,p):n\ge 5, k > n/2,p <2.3\}$. The loss can be removed if $V\in L^q(M)$ with $q>n/2$.
\end{theorem}
 
When $V\equiv 0$, the bounds are essentially sharp on the standard sphere. For the restriction to curved hypersurfaces, see the forthcoming Theorem \ref{mainthmcurv} for the improvements. We also obtain improvements on negatively curved manifolds and flat tori. 
\begin{theorem}\label{thmneg}
	Let $M$ be a negatively curved manifold of dimension $n\ge2$. When $n=3$, we assume it has constant negative sectional curvatures. Let $\Sigma\subset M$ be a geodesic segment. Let $V\in L^{n/2}(M)\cap \mathcal{K}(M)$. Then we have for all $2\le p\le \infty$
		\begin{equation}\label{SPVbdneg1}
		\|e_\la\|_{ L^p(\Sigma)}\ls \la^{\delta(1,p)}(\log\la)^{-\frac12}\|e_\la\|_{L^2(M)}.
	\end{equation}
\end{theorem}
This type of improvements for Laplacian eigenfunctions were investigated in a series of works by Sogge-Zelditch \cite{sogzel}($n=2$),  Chen-Sogge \cite{chensogge}($n=2,3$), Chen \cite{chen}($n\ge2$), Xi-Zhang \cite{xz}($n=2$), Blair \cite{mb}($n=2,3$), and  Zhang \cite{zhang}($n=3$). These are related to  the Kakeya-Nikodym tube estimate of eigenfunctions,  see a series of works by Blair and Sogge in \cite{sogge2011,bs1,bs2,bs3,bs4}.

Let $\mathbb{T}^n=\mathbb{R}^n/(2\pi\mathbb{Z}^n)$. Bourgain-Rudnick \cite{br2012} established uniform $L^2$ bounds on the restriction of Laplacian eigenfunctions on the flat tori $\mathbb{T}^2$ and $\mathbb{T}^3$ to curved hypersurfaces. Recently, Huang-Zhang \cite{hzapde} characterized the $L^2$ bounds on the restriction of toral eigenfunctions to totally geodesic submanifolds and obtained uniform bounds for rational hyperplanes. We partially extend  these  results to singular potentials.
\begin{theorem}\label{BRthm}
Let $V\in L^2(\mathbb{T}^2)$. If $\Sigma\subset\mathbb{T}^2$ is a curve segment with nonvanishing geodeisc curvature or a  segment of a closed geodesic, then for all $\la$ we have a uniform  bound 
\begin{equation}\label{tor1}
	\|e_\la\|_{ L^2(\Sigma)}\ls \|e_\la\|_{L^2(\mathbb{T}^2)}.
\end{equation}
 If $\Sigma$ is a geodesic segment, then for all $\la$ we have
\begin{equation}\label{tor2}
	\|e_\la\|_{ L^2(\Sigma)}\ls \sqrt{\log\la} \|e_\la\|_{L^2(\mathbb{T}^2)}.
\end{equation}
\end{theorem}
The conjectural bound for  \eqref{tor2} is a uniform constant, while it is equivalent to  the currently open question of whether on the
circle $|x| = \la$, the number of lattice points on an arc of size $\la^{\frac12}$ admits a uniform
bound. The condition  $V\in L^2$ is natural, and it is  related to the uniform $L^4$  bounds by Cooke \cite{cook} and Zygmund \cite{zyg}. The same condition also appears in  Bourgain-Burq-Zworski \cite{BBZ}.

The bounds of period integrals over submanifolds have been established by Good \cite{good}, Hejhal \cite{Hej}, Zelditch \cite{zeld1992}, Reznikov \cite{rez2015} and Chen-Sogge \cite{chensogge0}, etc. We prove  this type of bounds for singular potentials.

\begin{theorem}\label{mainthm2} Let $V\in L^{n/2}(M)\cap \mathcal{K}(M)$. Then for all $n,k$ we have
	\begin{equation}\label{PVbdper}
	\Big|\int_\Sigma e_\la d\sigma\Big| \ls \la^{\frac{n-k-1}2}\|e_\la\|_{L^2(M)},
	\end{equation}
albeit with a potential logarithmic loss  when $(n,k)\in \{(n,k):n\ge 9, n/2 < k \le n-4\}$. The loss can be removed if $V\in L^q(M)$ with $q>n/2$. Here   $d\sigma$  is the volume measure on $\Sigma$ induced by the Riemannian metric on $M$.
\end{theorem}
When $V\equiv0$,  the bounds are sharp on the sphere. We also obtain improvements on negatively curved manifolds. For instance, we state a two dimensional result.

\begin{theorem}\label{pineg}Let $M$ be a negatively curved surface. Let $\Sigma\subset M$ be any closed curve. Let  $V\in \mathcal{K}(M)$. Then we have
	\begin{equation}\label{PVbdperneg}
		\Big|\int_\Sigma e_\la d\sigma\Big| \ls (\log\la)^{-\frac12}(\log\log\la)\|e_\la\|_{L^2(M)}.
	\end{equation}
	The $\log\log\la$ loss can be removed if $V\in L^q(M)$ with $q>1$.
\end{theorem}
This type of improvements for Laplacian eigenfunctions are due to Sogge-Xi-Zhang \cite{sxz2017}, Wyman \cite{wym2019}, and Canzani--Galkowski \cite{CG21}, while the conjectural bound is $O(\la^{-\frac12+\eps})$ for any $\eps>0$, see Reznikov \cite{rez2015}. The conditions on $M$ may be significantly relaxed; see, e.g., Sogge-Xi-Zhang \cite{sxz2017}, Canzani--Galkowski--Toth \cite{cgt2018}, Canzani--Galkowski \cite{cgduke,CG21}, and Wyman \cite{wym2017,wym2019a,wym2019, wym2020}. 

\subsection{Proof methods}
To obtain eigenfunction bounds for $H_V$, a powerful tool is the second resolvent formula used by Blair-Sire-Sogge \cite{BSS} and Blair-Huang-Sire-Sogge \cite{BHSS}. Recently, Blair-Park \cite{BP,BP2} used this type of perturbative argument  to obtain  estimates for eigenfunctions restricted to  submanifolds of codimension 1 and 2, while higher codimension analogues were open. This  argument focuses on estimating the difference  between the resolvent operators for perturbed and unperturbed cases, where the differences only contribute to the error terms in the
main theorems.  However, a limitation of this method is that it cannot handle submanifolds of all codimensions. Specifically, it is easy to see that $((\la+i)^2+\Delta_g)^{-1}$ is not  $L^2(M)\to L^2(\Sigma)$ bounded when $k\le n-4$.

A natural idea  is to replace the resolvent operator by the spectral projection operator in the perturbative argument, since both of them can reproduce eigenfunctions. However, the new difficulty in the perturbative argument is to handle the  difference between the spectral projection operators for perturbed and unperturbed cases. Indeed, we use the  Duhamel principle for the wave equation to represent the spectral projection operators of $H_V$.   By the spectral theorem, we split the lower and higher frequencies respect to $\la$ in the Fourier expansion. We need to carefully handle several different cases regarding the interactions of different frequency regimes. This is the technical part of our proof.

We give some remarks on  our new method. First,  we use a bootstrap argument involving induction on the dimensions of the submanifolds to utilize the Kato class condition. This is the crucial idea that enables us to handle submanifolds of all codimensions (see Subsection \ref{hcd}). Moreover, an advantage of the method is that it is particularly useful to handle spectral projection operators with small windows. For instance, we can handle the spectral projection on the window $[\la,\la+\la^{-1}]$ to prove Theorem \ref{BRthm} on  flat tori. The method is expected to be useful in more general settings.

\subsection{Paper structure}
The paper is structured as follows. In Section 2, we presents some basic facts and related results we will use in our argument. In Section 3, we presents the main argument for the proof and we deal with higher codimensions in Subsection 3.1. In Section 4, we remove the logarithmic loss in Case 2 and finish the proof of Theorems \ref{mainthm}, \ref{thmneg}, \ref{mainthm2}, \ref{pineg}. In Section 5, we obtain some improvements on curved hypersurfaces. In Section 6, we prove Theorem \ref{BRthm} on the flat tori. In Section 7, we discuss potential improvements on oscillatory integrals.

\subsection{Notations}	Throughout this paper, $X\ls Y$ means $X\le CY$  for some positive constants $C$. If $X\ls Y$ and $Y\ls X$, we denote $X\approx Y$. If $x$ is in a small neighborhood of $x_0$, we denote $x\sim x_0$. 

\noindent\textbf{Acknowledgments.}
The authors would like to thank Nicolas Burq  for  some helpful comments. X.H. is partially supported by NSF DMS-2452860 and the Simons Foundation. X.W. is partially supported by the Fundamental Research Funds for the Central Universities Grant No. 531118010864 from Hunan University. C.Z. is partially supported by National Key R\&D Program of China No. 2024YFA1015300 and NSFC Grant No. 12371097. 
\section{Preliminary}
We shall assume throughout that the potentials $V$ are real-valued and $V\in \mathcal{K}(M)$, which is the Kato class. It is all $V\in L^1(M)$ satisfying  
\[\lim_{\delta\to 0}\sup_{x\in M}\int_{d_g(y,x)<\delta}|V(y)|W_n(d_g(x,y))dy=0,\]
where \[W_n(r)=\begin{cases}r^{2-n},\quad\quad\quad\quad n\ge3\\
	\log(2+r^{-1}),\ \ n=2\end{cases}\]
and $d_g$, $dy$ denote geodesic distance, the volume element on $(M,g)$.
Note that by H\"older inequality, we have $L^{q}\subset \mathcal{K}(M)\subset L^1(M)$ for all $q>\frac n2$. The Kato class $\mathcal{K}(M)$ and $L^{n/2}(M)$ share the same critical scaling
behavior, while neither one is contained in the other one for $n\ge3$. For instance, singularities of the type $|x|^{-\alpha}$ for $\alpha<2$ are allowed for both classes.

The Kato class is the ``minimal condition'' to ensure that $H_V$ is essentially self-adjoint and bounded from below, and  eigenfunctions of $H_V$ are bounded.  Since $M$ is compact, the spectrum of $H_V$ is discrete.  Also, the associated
eigenfunctions are continuous. The following Gaussian heat kernel bounds holds for all $x,y\in M$
\begin{equation}\label{heat}
	|e^{-tH_V}(x,y)|\le C t^{-\frac n2}e^{-cd_g(x,y)^2/t},\ 0<t\le 1.
\end{equation}
The constants $C$ and $c$ are positive, and they can be independent of $V$. See e.g. Sturm \cite{sturm}, Huang-Wang-Zhang \cite{HWZ24} for a detailed proof.

  After possibly adding a constant to $V$ we may assume throughout that $H_V$ is bounded below by one. 
We shall write the spectrum
of $\sqrt{H_V}$  as $\{\tau_k\}_{k=1}^\infty,$
where the eigenvalues are arranged in increasing order and we account for multiplicity.  For each $\tau_k$ there is an 
eigenfunction $e_{\tau_k}\in \text{Dom }(H_V)$ (the domain of $H_V$) so that
\begin{equation}\label{1.6}
	H_Ve_{\tau_k}=\tau^2_ke_{\tau_k},\ \ {\rm and}\  \ \int_M |e_{\tau_k}(x)|^2 \, dx=1.
\end{equation} 
Moreover, we shall let $	H^0=-\Delta_g$
be the unperturbed operator.  The corresponding eigenvalues and associated $L^2$-normalized
eigenfunctions are denoted by $\{\lambda_j\}_{j=1}^\infty$ and $\{e^0_j\}_{j=1}^\infty$, respectively so
that
\begin{equation}\label{1.9}
	H^0e^0_j=\lambda^2_j e^0_j, \quad \text{and }\, \, 
	\int_M |e^0_j(x)|^2 \, dx=1.
\end{equation}
Both $\{e_{\tau_k}\}_{k=1}^\infty$ and $\{e^0_j\}_{j=1}^\infty$ are orthonormal bases for $L^2(M)$. Let $P^0=\sqrt{H^0}$ and $P_V=\sqrt{H_V}$.  Let $\la\ge10$ and let $\1_{I}(\tau)$ be the indicator function of the interval $I$. We can define the spectral projection operator $\1_{I}(P^0)$ and $\1_{I}(P_V)$ by the spectral theorem.

Suppose $V\in \mathcal{K}(M)\cap L^{n/2}(M)$. Blair-Sire-Sogge \cite{BSS} and Blair-Huang-Sire-Sogge \cite{BHSS} established the following sharp $L^p$ bounds 
\begin{equation}\label{BSSLp}
	\|\1_{[\la,\la+1]}(P_V)\|_{L^2(M)\to L^p(M)}\ls\la^{\sigma(p)},
\end{equation} with Sogge's exponent
\begin{equation}\label{sog}
	\sigma(p)=\begin{cases}
		\frac{n-1}2-\frac np,\ \ \ \ \ \ \ q_c\le p\le \infty\\
		\frac{n-1}2(\frac12-\frac1p),\ \ \ \ 2\le p<q_c,
	\end{cases}
\end{equation}
where $q_c=\frac{2n+2}{n-1}$.
This extends Sogge's seminal work \cite{sogge88}. Moreover, for $n\ge3$, if $2\le p\le \frac{2n}{n-4}$ when $n\ge5$ and $2\le p<\infty$ when $n=3,4$, the requirement on potential can be relaxed to $V\in L^{n/2}(M)$. Furthermore, \eqref{BSSLp} implies that for all $\delta\ge1$
\begin{equation}\label{b1}
	\|\1_{[\la,\la+\delta]}(P_V)\|_{L^2(M)\to L^p(M)}\ls \la^{\sigma(p)}\delta^\frac12.
\end{equation}

On negatively curved manifolds, for $\eps=(\log\la)^{-1}$,  we have for all $p\ge q_c$
\begin{equation}\label{BSSLpneg}
	\|\1_{[\la,\la+\eps]}(P_V)\|_{L^2(M)\to L^p(M)}\ls\la^{\sigma(p)}\eps^{\frac12}.
\end{equation}
When $V\equiv0$, it is due to B\'erard \cite{be} ($p=\infty$), Hassell-Tacy \cite{hstc} ($p>q_c$) and Huang Sogge \cite{hs24} ($p=q_c$). It was established for $V\in \mathcal{K}(M)\cap L^{n/2}(M)$ by Blair-Huang-Sire-Sogge \cite{BHSS}. Moreover, \eqref{BSSLpneg} implies that for all $\delta\ge(\log\la)^{-1}$ and  $p\ge q_c$
\begin{equation}\label{b2}
		\|\1_{[\la,\la+\delta]}(P_V)\|_{L^2(M)\to L^p(M)}\ls \la^{\sigma(p)}\delta^\frac12.
\end{equation}

Next, we recall some estimates for the restriction of Laplacian  eigenfunctions to submanifolds. We have 
\begin{equation}
	\|\1_{[\la,\la+1]}(P^0)\|_{L^2(M)\to L^p(\Sigma)}\ls \la^{\delta(k,p)}
\end{equation}
albeit with a potential $(\log\lambda)^{\frac{1}{2}}$ loss when $(k,p)=(n-2,2)$. 

Let $M$ be a negatively curved manifold of dimension $n\ge2$. When $n=3$, we assume it has constant negative sectional curvatures. Let $\Sigma\subset M$ be a geodesic segment.  Then for $\eps=(\log\la)^{-1}$ we have 
\begin{equation}
	\|\1_{[\la,\la+\eps]}(P^0)\|_{ L^2(M)\to L^p(\Sigma)}\ls \la^{\delta(1,p)}(\log\la)^{-\frac12}.
\end{equation}
We remark that when $n=2$ only a $(\log\la)^{-\frac14}$ gain was obtained in  \cite{xz,mb} for non-positively curved manifolds, but it can be improved to  $(\log\la)^{-\frac12}$ on negatively curved manifolds by using Günther's comparison theorem, see \cite[Remark 1]{zhang}. When $n=3$, the constant negative curvatures condition is needed to get improvements, see \cite{chensogge,mb,zhang}. 

For period integrals over submanifolds, we have 
\begin{equation}
	\Big|\int_\Sigma\1_{[\la,\la+1]}(P^0)fd\sigma\Big|\ls \la^{\frac{n-k-1}2}\|f\|_{L^2(M)}.
\end{equation}
Remarkable improvements on negatively curved manifolds were obtained in \cite{chensogge0,sxz2017,wym2017,wym2019,wym2019a,wym2020,cgt2018,cgduke}. In particular, for any closed curve $\Sigma$ on negatively curved surfaces and $\eps=(\log\la)^{-1}$, we have 
\begin{equation}
		\Big|\int_\Sigma\1_{[\la,\la+\eps]}(P^0)fd\sigma\Big|\ls (\log\la)^{-\frac12}\|f\|_{L^2(M)}.
\end{equation}

\section{Main argument}\label{MG}
For $0<\eps\le 1$, let $\1_\la=\1_{[\la-\eps,\la+\eps]}$, 
$\1_{\la,\ell}(s)=\1_{|s-\la|\in (2^\ell,2^{\ell+1}]}(s)$, $\1_{\le2\la}=\1_{(-\infty,2\la]}$. To prove Theorem \ref{mainthm} and Theorem \ref{mainthm2}, we fix $\eps=1$. To prove Theorem \ref{thmneg} and Theorem \ref{pineg}, we fix $\eps=(\log\la)^{-1}$.  We shall use the argument in this section to prove these theorems. 

In the following, we denote  $p_c=\frac{2n}{n-1}$ when $k=n-1$ and $p_c=2$ when $k\le n-2$. These are the endpoints in the restriction bounds. We mainly focus on the endpoint estimates at $p=p_c$, as one can easily obtain other estimates by the same argument or by interpolation.	 Note that $p=2$ is also an endpoint when $k=n-1$, and we shall handle it independently in our proof. Let $\|f\|_X$ be the norms $\|f\|_{L^{p_c}(\Sigma)}$, $\|f\|_{L^2(\Sigma)}$ or the semi-norm $|\int_\Sigma fd\sigma|$. Suppose 
\begin{equation}\label{base}
	\|\1_\la(P^0)\|_{L^{2}(M)\to X}\ls A.
\end{equation}
Our goal is to prove
\begin{equation}\label{goal}
	\|\1_\la(P_V)\|_{L^{2}(M)\to X}\ls A,
\end{equation}
	except for some log loss in certain cases.
 We use the short notation $A$ in many cases where we do not need its explicit form, and the parameters are fixed.  
 
 Fix a nonnegative   function 
 $\chi \in C_0^\infty(\mathbb{R})$, such that  supp${\chi} \subset (-\eps,\eps)$. Let $\la\ge10$. Let $\chi_\la(s)=\chi(\la-s)$. By \eqref{base}, we have $$\|\chi_\la(P^0)\|_{L^{2}(M)\to X}\ls A.$$
 For $s\ge0$, we have $\chi(\la+s)=0$, and then
 \[\chi(\la-s)=\chi(\la-s)+\chi(\la+s)=\frac1{\pi}\int \hat\chi(t)e^{it\la}\cos (ts)dt.\]
 By Duhamel's principle and the spectral theorem, we can calculate the difference between the wave kernel and its perturbation as in \cite{hsweyl}, \cite{hz2021}, \cite{hz2023}
 \begin{align*}\cos tP_V(x,y)&-\cos tP^0(x,y)\\
 	&=-\sum_{\la_j}\sum_{\tau_k}\int_M\int_0^t\frac{\sin(t-s)\la_j}{\la_j}\cos s\tau_k\  e_j^0(x)e_j^0(z)e_{\tau_k}(z)e_{\tau_k}(y)V(z)dzds\\
 	&=\sum_{\la_j}\sum_{\tau_k}\int_M\frac{\cos t\la_j-\cos t\tau_k}{\la_j^2-\tau_k^2} e_j^0(x)e_j^0(z)e_{\tau_k}(z)e_{\tau_k}(y)V(z)dz.\end{align*}
 So we have
 \[\chi_\la(P_V)(x,y)-\chi_\la(P^0)(x,y)=\sum_{\la_j}\sum_{\tau_k}\int_M\frac{\chi_\la(\la_j)-\chi_\la(\tau_k)}{\la_j^2-\tau_k^2} e_j^0(x)e_j^0(z)e_{\tau_k}(z)e_{\tau_k}(y)V(z)dz\]
 Let $m(\la_j,\tau_k)=\frac{\chi_\la(\la_j)-\chi_\la(\tau_k)}{\la_j^2-\tau_k^2}$. 
 It suffices to estimate the $L^2(M)\to X$ bound of the operator associated with the kernel
 \begin{equation}\label{mainest}K(x,y)=\sum_{\la_j}\sum_{\tau_k}\int_Mm(\la_j,\tau_k) e_j^0(x)e_j^0(z)e_{\tau_k}(z)e_{\tau_k}(y)V(z)dz.\end{equation}

 By the support property of $m(\la_j,\tau_k)$, we need to consider five cases.

\begin{enumerate}
	\item $|\tau_k-\la|\le \eps$, $|\la_j-\la|\le \eps$.
	\item $|\tau_k-\la|\le \eps$, $|\la_j-\la|\in(2^\ell,2^{\ell+1}]$, $\eps\le 2^\ell\le \la$.
	\item $|\la_j-\la|\le\eps$, $|\tau_k-\la|\in(2^\ell,2^{\ell+1}]$, $\eps\le 2^\ell\le \la$.
	\item $|\la_j-\la|\le \eps$, $\tau_k>2\la$.
	\item $|\tau_k-\la|\le \eps$, $\la_j>2\la$.
	
\end{enumerate}
 
 Cases 1, 3, 4 are relatively straightforward and their contributions are $O(A)$ as desired. Case 2 will give a log loss for critical potentials, but we shall remove it in the next section by the resolvent method. Case 5 is more involved and we shall use a bootstrap argument involving an induction on the dimensions of the submanifolds.

In the following, we fix $q= \frac n2$, and fix $\frac1{p_0}\in [\frac{n+3}{2n+2}-\frac 2n,\frac{n-1}{2n+2}]$ and $\frac1{q_0}=\frac 1{p_0'}-\frac 1q$ when $n\ge3$, and $p_0=q_0=\infty$ when $n=2$. So we always have $q_0\ge\frac{2n+2}{n-1}$ and then  $\sigma(p_0)+\sigma(q_0)=1$. In general, for $n\ge3$ we will see that $\frac1{p_0}=\frac{n+3}{2n+2}-\frac2n$ and $q_0=\frac{2n+2}{n-1}$ is the best choice.

$ $

\noindent	\textbf{Case 1.}  $|\tau_k-\la|\le \eps$, $|\la_j-\la|\le \eps$.
	
	In this case, for $|s-\la|\le\eps$ we have
	\[|m(\la_j,s)|+\eps|\partial_s m(\la_j,s)|\ls (\la\eps)^{-1}.\]Then 
	\begin{align*}&\sum_{|\la_j-\la|\le\eps}\sum_{|\tau_k-\la|\le \eps}\int_Mm(\la_j,\tau_k)e_j^0(x)e_j^0(z)e_{\tau_k}(z)e_{\tau_k}(y)V(z)dz\\
	&= \sum_{|\la_j-\la|\le\eps}\sum_{|\tau_k-\la|\le \eps}\int_M\int_{\la-\eps}^{\la+\eps}\partial_sm(\la_j,s)\1_{[\la-\eps,\tau_k]}(s)e_j^0(x)e_j^0(z)e_{\tau_k}(z)e_{\tau_k}(y)V(z)dzds\\
	&\ \ \ +\sum_{|\la_j-\la|\le\eps}\sum_{|\tau_k-\la|\le \eps}\int_Mm(\la_j,\la-\eps)e_j^0(x)e_j^0(z)e_{\tau_k}(z)e_{\tau_k}(y)V(z)dz\\
	&=K_1(x,y)+K_2(x,y).
\end{align*}

We handle $K_2$ first. For any $f\in L^{2}(M)$, we have
\begin{align*}
	\|K_2f\|_{X}&=\|\1_\la(P^0)m(P^0,\la-\eps)(V\cdot \1_\la(P_V)f)\|_{X}\\
	&\ls A\|\1_\la(P^0)m(P^0,\la-\eps)(V\cdot \1_\la(P_V)f)\|_{L^2}\\
	&\ls A(\la\eps)^{-1}\|\1_\la(P^0)(V\cdot \1_\la(P_V)f)\|_{L^2}\\
	&\ls A(\la\eps)^{-1}\la^{\sigma(p_0)}\eps^{1/2}\|V\cdot\1_\la(P_V)f\|_{L^{p_0'}}\\
	&\ls A(\la\eps)^{-1}\la^{\sigma(p_0)}\eps^{1/2}\|V\|_{L^q}\|\1_\la(P_V)f\|_{L^{q_0}}\\
	&\ls A(\la\eps)^{-1}\la^{\sigma(p_0)}\eps^{1/2}\la^{\sigma(p_0)}\eps^{1/2}\|V\|_{L^q}\|f\|_{L^{2}}\\
	&=A\|V\|_{L^q}\|f\|_{L^{2}}.
\end{align*}
In the first inequality we used the identity $\1_\la(P^0)\circ\1_\la(P^0) = \1_\la(P^0)$, which will be used again later without further explanation. We also used \eqref{BSSLp} or \eqref{BSSLpneg} in the fourth and sixth inequalities. The approach for $K_1$ is similar.

$ $
		
\noindent \textbf{Case 2.}  $|\tau_k-\la|\le \eps$, $|\la_j-\la|\in(2^\ell,2^{\ell+1}]$, $\eps\le 2^\ell\le \la$.

Let $\psi\in C_0^\infty(\mathbb{R})$ satisfy $\psi(t)=1$ if $|t|\le  2$ and $\psi(t)=0$ if  $|t|>3$. We split the  $\la_j$-frequencies by the cutoff function $\psi(\la_j/\la)$. When $|\la_j-\la| \in(2^\ell,2^{\ell+1}]$, we have $m(\la_j,\tau_k)=\frac{-\chi_\la(\tau_k)}{\la_j^2-\tau_k^2}\psi(\la_j/\la)$, and for $|s-\la|\le \eps$ 
\[|m(\la_j,s)|+\eps|\partial_s m(\la_j,s)|\ls \la^{-1}2^{-\ell}.\]
We can use the same argument as Case 1 to handle 
\begin{align*}&\sum_{|\la_j-\la|\in(2^\ell,2^{\ell+1}]}\sum_{|\tau_k-\la|\le\eps}\int_Mm(\la_j,\tau_k)e_j^0(x)e_j^0(z)e_{\tau_k}(z)e_{\tau_k}(y)V(z)dz\\
	&= \sum_{|\la_j-\la|\in(2^\ell,2^{\ell+1}]}\sum_{|\tau_k-\la|\le\eps}\int_M\int_{\la-\eps}^{\la+\eps}\partial_sm(\la_j,s)\1_{[\la-\eps,\tau_k]}(s)e_j^0(x)e_j^0(z)e_{\tau_k}(z)e_{\tau_k}(y)V(z)dzds\\
	&\ \ \ +\sum_{|\la_j-\la|\in(2^\ell,2^{\ell+1}]}\sum_{|\tau_k-\la|\le\eps}\int_Mm(\la_j,\la-\eps)e_j^0(x)e_j^0(z)e_{\tau_k}(z)e_{\tau_k}(y)V(z)dz\\
	&=K_{1,\ell}(x,y)+K_{2,\ell}(x,y).\end{align*}

We handle $K_{2,\ell}$ first. For any $f\in L^{2}(M)$, by \eqref{b1} and \eqref{b2} we have
\begin{align*}
	\|K_{2,\ell}f\|_{X}&=\|\1_{\la,\ell}(P^0)m(P^0,\la-\eps)(V\cdot \1_\la(P_V)f)\|_{X}\\
	&\ls A(2^\ell/\eps)^{1/2}\|\1_{\la,\ell}(P^0)m(P^0,\la-\eps)(V\cdot \1_\la(P_V)f)\|_{L^2}\\
	&\ls A(2^\ell/\eps)^{1/2}\la^{-1}2^{-\ell}\|\1_{\la,\ell}(P^0)(V\cdot \1_\la(P_V)f)\|_{L^2}\\
	&\ls A(2^\ell/\eps)^{1/2}\la^{-1}2^{-\ell}\la^{\sigma(p_0)}2^{\ell/2}\|V\cdot\1_\la(P_V)f\|_{L^{p_0'}}\\
	&\ls A(2^\ell/\eps)^{1/2}\la^{-1}2^{-\ell}\la^{\sigma(p_0)}2^{\ell/2}\|V\|_{L^q}\|\1_\la(P_V)f\|_{L^{q_0}}\\
	&\ls A(2^\ell/\eps)^{1/2}\la^{-1}2^{-\ell}\la^{\sigma(p_0)}2^{\ell/2}\la^{\sigma(q_0)}\eps^{1/2}\|V\|_{L^q}\|f\|_{L^{2}}\\
		&=A\|V\|_{L^q}\|f\|_{L^{2}}.
\end{align*}
The method to handle $K_{1,\ell}$ is similar. Summing over $\ell$ gives the bound $A\log\la$. We shall remove the log loss in the next section by resolvent method. Clearly, the log loss can be removed if we assume $q>n/2$ in the argument above.

$ $

\noindent \textbf{Case 3.}  $|\la_j-\la|\le\eps$, $|\tau_k-\la|\in(2^\ell,2^{\ell+1}]$, $\eps\le 2^\ell\le \la$.

In this case,  $m(\la_j,\tau_k)=\frac{\chi_\la(\la_j)}{\la_j^2-\tau_k^2}$, and for $|s-\la|\le \eps$ we have
\[|m(s,\tau_k)|+\eps|\partial_s m(s,\tau_k)|\ls \la^{-1}2^{-\ell}.\]
We can use the same argument as Case 1 to handle 
\begin{align*}&\sum_{|\tau_k-\la|\in(2^\ell,2^{\ell+1}]}\sum_{|\la_j-\la|\le\eps}\int_Mm(\la_j,\tau_k)e_j^0(x)e_j^0(z)e_{\tau_k}(z)e_{\tau_k}(y)V(z)dz\\
	&= \sum_{|\tau_k-\la|\in(2^\ell,2^{\ell+1}]}\sum_{|\la_j-\la|\le\eps}\int_M\int_{\la-\eps}^{\la+\eps}\partial_sm(s,\tau_k)\1_{[\la-\eps,\la_j]}(s)e_j^0(x)e_j^0(z)e_{\tau_k}(z)e_{\tau_k}(y)V(z)dzds\\
	&\ \ \ +\sum_{|\tau_k-\la|\in(2^\ell,2^{\ell+1}]}\sum_{|\la_j-\la|\le\eps}\int_Mm(\la-\eps,\tau_k)e_j^0(x)e_j^0(z)e_{\tau_k}(z)e_{\tau_k}(y)V(z)dz\\
	&=K_{1,\ell}(x,y)+K_{2,\ell}(x,y).\end{align*}

We handle $K_{2,\ell}$ first. For any $f\in L^{2}(M)$, we have
\begin{align*}
	\|K_{2,\ell}f\|_{X}&=\|\1_{\la}(P^0)m(\la-\eps,P^0)(V\cdot \1_{\la,\ell}(P_V)f)\|_{X}\\
	&\ls A\|\1_{\la}(P^0)m(\la-\eps,P^0)(V\cdot \1_{\la,\ell}(P_V)f)\|_{L^2}\\
	&\ls A\la^{-1}2^{-\ell}\|\1_{\la}(P^0)(V\cdot \1_{\la,\ell}(P_V)f)\|_{L^2}\\
	&\ls A\la^{-1}2^{-\ell}\la^{\sigma(p_0)}\eps^{1/2}\|V\cdot\1_{\la,\ell}(P_V)f\|_{L^{p_0'}}\\
	&\ls A\la^{-1}2^{-\ell}\la^{\sigma(p_0)}\eps^{1/2}\|V\|_{L^q}\|\1_{\la,\ell}(P_V)f\|_{L^{q_0}}\\
	&\ls A\la^{-1}2^{-\ell}\la^{\sigma(p_0)}\eps^{1/2}\la^{\sigma(q_0)}2^{\ell/2}\|V\|_{L^q}\|f\|_{L^{2}}\\
	&=A\eps^{1/2}2^{-\ell/2}\|V\|_{L^q}\|f\|_{L^{2}}.
\end{align*}
The method to handle $K_{1,\ell}$ is similar. Summing over $\ell$ gives the desired bound $A$.

$ $

\noindent \textbf{Case 4.} $|\la_j-\la|\le \eps$, $\tau_k>2\la$.

In this case, $m(\la_j,\tau_k)=\frac{\chi_\la(\la_j)}{\la_j^2-\tau_k^2}$. 

We write
\begin{align*}
	\frac{\chi_\lambda(\la_j)}{\tau_k^2-\lambda_j^2} & =\int_0^{\infty} \chi_\lambda(\la_j) e^{-t(\tau_k^2-\lambda_j^2)} d t \\
	& =\int_0^{\lambda^{-2}} \chi_\lambda(\la_j) e^{-t(\tau_k^2-\lambda_j^2)} d t+\frac{\chi_\lambda(\la_j) e^{-\lambda^{-2}(\tau_k^2-\lambda_j^2)}}{\tau_k^2-\lambda_j^2}\\
	&= m_1(\la_j,\tau_k)+ m_2(\la_j,\tau_k).
\end{align*}
We first handle $m_2$. Split the interval $(2\la,\infty)=\cup_{\ell=1}^\infty I_\ell$ with  $I_\ell=(2^\ell \la,2^{\ell+1}\la]$.
For $\tau_k\in I_\ell$ and $|s-\la|\le\eps$, we have
\[| m_2(s,\tau_k)|+\eps|\partial_s  m_2(s,\tau_k)|\ls \la^{-2}2^{-N\ell}.\]
Then we can use the same argument as Case 1 to obtain
\begin{align*}&\sum_{|\la_j-\la|\le\eps}\sum_{\tau_k\in I_\ell}\int_M m_2(\la_j,\tau_k)e_j^0(x)e_j^0(z)e_{\tau_k}(z)e_{\tau_k}(y)V(z)dz\\
	&= \sum_{|\la_j-\la|\le\eps}\sum_{\tau_k\in I_\ell}\int_M\int_{\la-\eps}^{\la+\eps}\partial_sm_2(s,\tau_k)\1_{[\la-\eps,\la_j]}(s)e_j^0(x)e_j^0(z)e_{\tau_k}(z)e_{\tau_k}(y)V(z)dzds\\
	&\ \ \ +\sum_{|\la_j-\la|\le\eps}\sum_{\tau_k\in I_\ell}\int_Mm_2(\la-\eps,\tau_k)e_j^0(x)e_j^0(z)e_{\tau_k}(z)e_{\tau_k}(y)V(z)dz\\
	&=K_{1,\ell}(x,y)+K_{2,\ell}(x,y).\end{align*}

We handle $K_{2,\ell}$ first. For any $f\in L^{2}(M)$, we have
\begin{align*}
	\|K_{2,\ell}f\|_{X}&=\|\1_{\la}(P^0)m_2(P^0,\la-\eps)(V\cdot \1_{I_\ell}(P_V)f)\|_{X}\\
	&\ls A\|\1_{\la}(P^0)m_2(P^0,\la-\eps)(V\cdot \1_{I_\ell}(P_V)f)\|_{L^2}\\
	&\ls A\la^{-2}2^{-N\ell}\|\1_{\la}(P^0)(V\cdot \1_{I_\ell}(P_V)f)\|_{L^2}\\
	&\ls A\la^{-2}2^{-N\ell}\la^{\sigma(p_0)}\eps^{1/2}\|V\cdot\1_{I_\ell}(P_V)f\|_{L^{p_0'}}\\
	&\ls A\la^{-2}2^{-N\ell}\la^{\sigma(p_0)}\eps^{1/2}\|V\|_{L^q}\|\1_{I_\ell}(P_V)f\|_{L^{q_0}}\\
	&\ls A\la^{-2}2^{-N\ell}\la^{\sigma(p_0)}\eps^{1/2}(\la2^\ell)^{\sigma(q_0)}(\la2^\ell)^{1/2}\|V\|_{L^q}\|f\|_{L^{2}}\\
	&=A\la^{-\frac12}2^{-N'\ell}\eps^{1/2}\|V\|_{L^q}\|f\|_{L^{2}}.
\end{align*}
The method to handle $K_{1,\ell}$ is similar.

Next we handle $m_1$.	As before, we split the sum $\sum_{\tau_k>2\la}$ into the difference of the complete sum 
\begin{equation}\label{compsum}\sum_{|\la_j-\la|\le\eps}\sum_{\tau_k}\int_M m_1(\la_j,\tau_k)e_j^0(x)e_j^0(z)V(z)e_{\tau_k}(z)e_{\tau_k}(y)dz\end{equation}
and the  partial sum
\begin{equation}\label{partsum}\sum_{|\la_j-\la|\le\eps}\sum_{\tau_k\le 2\la}\int_M m_1(\la_j,\tau_k)e_j^0(x)e_j^0(z)V(z)e_{\tau_k}(z)e_{\tau_k}(y)dz.\end{equation} We first handle the partial sum. When $\tau_k\le 2\la$ and $|s-\la|\le\eps$, we have
\[| m_1(s,\tau_k)|+\eps|\partial_s  m_1(s,\tau_k)|\ls \la^{-2}.\]
Then we can use the same argument as Case 1 to handle
\begin{align*}&\sum_{|\la_j-\la|\le\eps}\sum_{\tau_k\le 2\la}\int_M m_1(\la_j,\tau_k)e_j^0(x)e_j^0(y)e_{\tau_k}(x)e_{\tau_k}(y)V(y)dy\\
	&= \sum_{|\la_j-\la|\le\eps}\sum_{\tau_k\le 2\la}\int_M\int_{\la-\eps}^{\la+\eps}\partial_sm_1(s,\tau_k)\1_{[\la-1,\la_j]}(s)e_j^0(x)e_j^0(z)e_{\tau_k}(z)e_{\tau_k}(y)V(z)dzds\\
	&\ \ \ +\sum_{|\la_j-\la|\le\eps}\sum_{\tau_k\le 2\la}\int_Mm_1(\la-\eps,\tau_k)e_j^0(x)e_j^0(z)e_{\tau_k}(z)e_{\tau_k}(y)V(z)dz\\
	&=K_1(x,y)+K_2(x,y).
\end{align*}

We handle $K_2$ first. For any $f\in L^{2}(M)$, we have
\begin{align*}
	\|K_2f\|_{X}&=\|\1_\la(P^0)(V\cdot \1_{\le2\la}(P_V)m_1(\la-\eps,P_V)f)\|_{X}\\
	&\ls A\|\1_\la(P^0)(V\cdot \1_{\le2\la}(P_V)m_1(\la-\eps,P_V)f)\|_{L^2}\\
	&\ls A\la^{\sigma(p_0)}\eps^{1/2}\|V\cdot \1_{\le2\la}(P_V)m_1(\la-\eps,P_V)f\|_{L^{p'_0}}\\
	&\ls A\la^{\sigma(p_0)}\eps^{1/2}\|V\|_{L^q}\|\1_{\le2\la}(P_V)m_1(\la-\eps,P_V)f\|_{L^{q_0}}\\
	&\ls A\la^{\sigma(p_0)}\eps^{1/2}\la^{\sigma(q_0)}\la^{1/2}\|V\|_{L^q}\|m_1(\la-\eps,P_V)f\|_{L^{2}}\\
	&\ls A\la^{\sigma(p_0)}\eps^{1/2}\la^{\sigma(q_0)}\la^{1/2}\la^{-2}\|V\|_{L^q}\|f\|_{L^{2}}\\
	&=A\la^{-\frac12}\eps^{1/2}\|V\|_{L^q}\|f\|_{L^{2}}.
\end{align*}
 The method to handle $K_1$ is similar.

To handle the complete sum \eqref{compsum}, we need the heat kernel bounds \begin{equation}\label{heatV}
	\|e^{-tH_V}\|_{L^p(M)\to L^q(M)}\ls t^{-\frac n2(\frac1p-\frac1q)},\ \ \text{if}\ 0<t\le1\ \text{and}\ 1\le p\le q\le \infty.
\end{equation}
This follows from \eqref{heat} and Young's inequality.
Then we have
\begin{align*}
	\int_0^{\la^{-2}}\|\1_\la(P^0)e^{-t\Delta_g}(V\cdot e^{-tH_V}f)\|_{X}dt&\ls A \int_0^{\la^{-2}}\|\1_\la(P^0)(V\cdot e^{-tH_V}f)\|_{L^{2}}dt\\
	&\ls A\la^{\sigma(p_0)}\eps^{1/2} \int_0^{\la^{-2}}\|V\cdot e^{-tH_V}f\|_{L^{p_0'}}dt\\
	&\ls A\la^{\sigma(p_0)}\eps^{1/2}\|V\|_{L^q} \int_0^{\la^{-2}}\| e^{-tH_V}f\|_{L^{q_0}}dt\\
	&\ls A\la^{\sigma(p_0)}\eps^{1/2}\|V\|_{L^q}\|f\|_{L^{2}} \int_0^{\la^{-2}}t^{-\frac n2(\frac1{2}-\frac1{q_0})}dt\\
	&\ls A\la^{\sigma(p_0)}\eps^{1/2}\|V\|_{L^q}\|f\|_{L^{2}}\la^{-2+n(\frac12-\frac1{q_0})}\\
	&= A\la^{-\frac12}\eps^{1/2}\|V\|_{L^q}\|f\|_{L^{2}}.
\end{align*}

$ $

\noindent \textbf{Case 5.} $|\tau_k-\la|\le \eps$, $\la_j>2\la$.

We first deal with the case $k > n-4$.
Recall that in Case 2, we split the frequencies by the cutoff function $\psi\in C_0^\infty(\mathbb{R})$ satisfying $\psi(t)=1$ if $|t|\le  2$ and $\psi(t)=0$ if  $|t|>3$. So now we need to deal with $m(\la_j,\tau_k)=\frac{-\chi_\la(\tau_k)}{\la_j^2-\tau_k^2}(1-\psi(\la_j/\la))$. 

We write
\begin{align*}
	\frac{\chi_\lambda(\tau_k)}{\lambda_j^2-\tau_k^2} & =\int_0^{\infty} \chi_\lambda(\tau_k) e^{-t(\lambda_j^2-\tau_k^2)} d t \\
	& =\int_0^{\lambda^{-2}} \chi_\lambda(\tau_k) e^{-t(\lambda_j^2-\tau_k^2)} d t+\frac{\chi_\lambda(\tau_k) e^{-\lambda^{-2}(\lambda_j^2-\tau_k^2)}}{\lambda_j^2-\tau_k^2}\\
	&:= m_1(\la_j,\tau_k)+ m_2(\la_j,\tau_k).
\end{align*}
We first handle $m_2$. Split the interval $(2\la,\infty)=\cup_{\ell=1}^\infty I_\ell$ with  $I_\ell=(2^\ell \la,2^{\ell+1}\la]$.
For $\la_j\in I_\ell$ and $|s-\la|\le\eps$, we have
\[| m_2(\la_j,s)|+\eps|\partial_s  m_2(\la_j,s)|\ls \la^{-2}2^{-N\ell},\ \  \forall N.\]
Then we can use the same argument as Case 1 to obtain
\begin{align*}&\sum_{|\tau_k-\la|\le\eps}\sum_{\la_j\in I_\ell}\int_M m_2(\la_j,\tau_k)e_j^0(x)e_j^0(z)e_{\tau_k}(z)e_{\tau_k}(y)V(z)dz\\
	&= \sum_{|\tau_k-\la|\le\eps}\sum_{\la_j\in I_\ell}\int_M\int_{\la-\eps}^{\la+\eps}\partial_sm_2(\la_j,s)\1_{[\la-\eps,\tau_k]}(s)e_j^0(x)e_j^0(z)e_{\tau_k}(z)e_{\tau_k}(y)V(z)dzds\\
	&\ \ \ +\sum_{|\tau_k-\la|\le\eps}\sum_{\la_j\in I_\ell}\int_Mm_2(\la_j,\la-\eps)e_j^0(x)e_j^0(z)e_{\tau_k}(z)e_{\tau_k}(y)V(z)dz\\
	&=K_{1,\ell}(x,y)+K_{2,\ell}(x,y).\end{align*}
We handle $K_{2,\ell}$ first. By \eqref{b1} we have
\[\|\1_{I_\ell}(P^0)\|_{L^2(M)\to X}\ls A 2^{\alpha\ell}(\la2^\ell)^\frac12,\]
where $\alpha=\delta(k,p_c)$ if $\|f\|_X=\|f\|_{L^{p_c}(\Sigma)}$,  and $\alpha=\delta(k,2)$ if $\|f\|_X=\|f\|_{L^2(\Sigma)}$ or  $|\int_\Sigma fd\sigma|$. 
For any $f\in L^{2}(M)$, we have
\begin{align*}
	\|K_{2,\ell}f\|_{X}&=\|\1_{I_\ell}(P^0)m_2(P^0,\la-1)(V\cdot \1_\la(P_V)f)\|_{X}\\
	&\ls A 2^{\alpha\ell}(\la2^\ell)^\frac12\|\1_{I_\ell}(P^0) m_2(P^0,\la-1)(V\cdot \1_\la(P_V)f)\|_{L^2}\\
	&\ls A 2^{\alpha\ell}(\la2^\ell)^\frac12\la^{-2}2^{-N\ell}\|\1_{I_\ell}(P^0)(V\cdot \1_\la(P_V)f)\|_{L^2}\\
	&\ls A 2^{\alpha\ell}(\la2^\ell)^\frac12 \la^{-2}2^{-N\ell}(\la 2^\ell)^{\sigma(p_0)+\frac12}\|V\cdot\1_\la(P_V)f\|_{L^{p_0'}}\\
	&\ls A 2^{\alpha\ell}(\la2^\ell)^\frac12 \la^{-2}2^{-N\ell}(\la 2^\ell)^{\sigma(p_0)+\frac12}\|V\|_{L^q}\|\1_\la(P_V)f\|_{L^{q_0}}\\
	&\ls A 2^{\alpha\ell}(\la2^\ell)^\frac12 \la^{-2}2^{-N\ell}(\la 2^\ell)^{\sigma(p_0)+\frac12}\la^{\sigma(q_0)}\eps^{1/2}\|V\|_{L^q}\|f\|_{L^{2}}\\
	&=A2^{-N_1\ell}\eps^{1/2}\|V\|_{L^q}\|f\|_{L^{2}}.
\end{align*}
The method to handle $K_{1,\ell}$ is similar.

Next, we handle $ m_1$.	 We split the sum $\sum_{\la_j>2\la}$ into the difference of the complete sum \begin{equation}\label{compsum0}\sum_{|\tau_k-\la|\le\eps}\sum_{\la_j}\int_M m_1(\la_j,\tau_k)e_j^0(x)e_j^0(z)e_{\tau_k}(z)e_{\tau_k}(y)V(z)dz\end{equation}
and the  partial sum
\begin{equation}\label{partsum0}\sum_{|\tau_k-\la|\le\eps}\sum_{\la_j\le 2\la}\int_M m_1(\la_j,\tau_k)e_j^0(x)e_j^0(z)e_{\tau_k}(z)e_{\tau_k}(y)V(z)dz.\end{equation} We first handle the partial sum. When $\la_j\le 2\la$ and $|s-\la|\le\eps$, we have
\[| m_1(\la_j,s)|+\eps|\partial_s  m_1(\la_j,s)|\ls \la^{-2}.\]
Then we can use the same argument as Case 1 to handle
\begin{align*}&\sum_{|\tau_k-\la|\le\eps}\sum_{\la_j\le 2\la}\int_M m_1(\la_j,\tau_k)e_j^0(x)e_j^0(y)e_{\tau_k}(x)e_{\tau_k}(y)V(y)dy\\
	&= \sum_{|\tau_k-\la|\le\eps}\sum_{\la_j\le 2\la}\int_M\int_{\la-\eps}^{\la+\eps}\partial_sm_1(\la_j,s)\1_{[\la-\eps,\tau_k]}(s)e_j^0(x)e_j^0(z)e_{\tau_k}(z)e_{\tau_k}(y)V(z)dzds\\
	&\ \ \ +\sum_{|\tau_k-\la|\le\eps}\sum_{\la_j\le 2\la}\int_Mm_1(\la_j,\la-\eps)e_j^0(x)e_j^0(z)e_{\tau_k}(z)e_{\tau_k}(y)V(z)dz\\
	&=K_1(x,y)+K_2(x,y).
\end{align*}
We handle $K_2$ first. For any $f\in L^{2}(M)$, we have
\begin{align*}
	\|K_2f\|_X&=\|\1_{\le 2\la}(P^0)m_1(P^0,\la-\eps)(V\cdot \1_\la(P_V)f)\|_{X}\\
	&\ls A\la^{1/2}\|\1_{\le 2\la}(P^0)m_1(P^0,\la-\eps)(V\cdot \1_\la(P_V)f)\|_{L^2}\\
	&\ls A\la^{1/2}\la^{-2}\|\1_{\le 2\la}(P^0)(V\cdot \1_\la(P_V)f)\|_{L^2}\\
	&\ls A\la^{1/2}\la^{-2}\la^{\sigma(p_0)+\frac12}\|V\cdot\1_\la(P_V)f\|_{L^{p_0'}}\\
	&\ls A\la^{1/2}\la^{-2}\la^{\sigma(p_0)+\frac12}\|V\|_{L^q}\|\1_\la(P_V)f\|_{L^{q_0}}\\
	&\ls A\la^{1/2}\la^{-2}\la^{\sigma(p_0)+\frac12}\la^{\sigma(q_0)}\eps^{1/2}\|V\|_{L^q}\|f\|_{L^{2}}\\	&=A\eps^{1/2}\|V\|_{L^q}\|f\|_{L^{2}}.
\end{align*}
The method to handle $K_1$ is similar.

To handle the complete sum \eqref{compsum0}, we  use the heat kernel Gaussian bounds to calculate the kernel of $m_1(P^0,s)$ with $|s-\la|\le\eps$
\begin{align*}|\sum_{\la_j} m_1(\la_j,s)e_j^0(x)e_j^0(y)|&\ls \int_0^{\la^{-2}}|\sum_{\la_j}e^{-t\la_j^2}e_j^0(x)e_j^0(y)|dt\\ \nonumber
	&\ls \int_0^{\la^{-2}} t^{-\frac n2}e^{-cd_g(x,y)^2/t}dt\\ \nonumber
	&\ls \begin{cases}
		\log(2+(\la d_g(x,y))^{-1})(1+\la d_g(x,y))^{-N},\quad n=2\\\nonumber
		d_g(x,y)^{2-n}(1+\la d_g(x,y))^{-N},\quad\quad\quad\quad\quad\ \  n\ge3
	\end{cases}\\
	&\ls W_n(d_g(x,y))(1+\la d_g(x,y))^{-N},\ \ \ \ \forall N.
\end{align*}
Similarly, for $|s-\la|\le\eps$ we also have \[|\partial_s m_1(P^0,s)(x,y)|\ls W_n( d_g(x,y))(1+\la d_g(x,y))^{-N},\ \ \ \ \forall N.\]
We write
\begin{align*}
	&\sum_{|\tau_k-\la|\le\eps}\sum_{\la_j}\int_M m_1(\la_j,\tau_k)e_j^0(x)e_j^0(z)e_{\tau_k}(z)e_{\tau_k}(y)V(z)dz\\
	&= \sum_{|\tau_k-\la|\le\eps}\sum_{\la_j}\int_M\int_{\la-\eps}^{\la+\eps}\partial_sm_1(\la_j,s)\1_{[\la-\eps,\tau_k]}(s)e_j^0(x)e_j^0(z)e_{\tau_k}(z)e_{\tau_k}(y)V(z)dzds\\
	&\ \ \ +\sum_{|\tau_k-\la|\le1}\sum_{\la_j}\int_Mm_1(\la_j,\la-\eps)e_j^0(x)e_j^0(z)e_{\tau_k}(z)e_{\tau_k}(y)V(z)dz\\
	&=K_{1}(x,y)+K_{2}(x,y).
\end{align*}
We only handle $K_2$, and $K_1$ can be handled similarly.  

Let $p_2=p_c$ if $\|f\|_X=\|f\|_{L^{p_c}(\Sigma)}$,  and $p_2=2$ if $\|f\|_X=\|f\|_{L^2(\Sigma)}$ or  $|\int_\Sigma fd\sigma|$. By Young's inequality, we obtain
\begin{align*}
	\|m_1(P^0,\la-\eps)(V\cdot \1_\la(P_V)f)\|_{X}&\ls \|m_1(P^0,\la-\eps)(V\cdot \1_\la(P_V)f)\|_{L^{p_2}(\Sigma)}\\
	&\ls  \la^{-2+\frac n{p_1'}-\frac k{p_2}}\|V\cdot \1_\la(P_V)f\|_{L^{p_1'}}\\
	&\ls \la^{-2+\frac n{p_1'}-\frac k{p_2}}\|V\|_{L^q}\|\1_\la(P_V)f\|_{L^{q_1}}\\
	&\ls \la^{-2+\frac n{p_1'}-\frac k{p_2}+\sigma(q_1)}\eps^{1/2}\|V\|_{L^q}\|f\|_{L^2}.
\end{align*}
Here we require that $p_1'\le p_2$, $ \frac n{p_1'}-\frac k{p_2}<2$ and $\frac1{p_1'}=\frac1q+\frac1{q_1}$ with $q_1\ge \frac{2n+2}{n-1}$ and $q=n/2$. 
For $k>n-4$,  we set  $\frac1{p_1'}=\max(\frac 2n,\frac 1{p_2})$, and $\frac1{q_1}=\max(0,\frac1{p_2}-\frac2n)\le \max(0,\frac{n-4}{2n})\le\frac{n-1}{2n+2}.$ Then $\sigma(q_1)=\frac{n-1}2-\frac n{q_1}$. So we obtain the desired exponent
\[-2+\frac n{p_1'}-\frac k{p_2}+\sigma(q_1)=\frac{n-1}2-\frac k{p_2}=\delta(k,p_2).\]
This completes the proof for the case $k>n-4$. 

\subsection{Higher codimensions}\label{hcd}
In this subsection, we only handle $n\ge5$ and $k\le n-4$ with $\eps=1$. In these cases, we have $p_c=2$ and $A=\la^{\frac{n-k-1}2}$. To be specific, we only work on the endpoint $L^2(\Sigma)$ norm in the following, though the argument can still work for non-endpoint $L^p(\Sigma)$ norms. We shall use a bootstrap argument involving an induction on the dimensions of the submanifolds.

Fix $\ell \in \mathbb{Z}$ with $1\le 2^\ell\ll\la$. Let $\{x_j\}$ be a maximal $\la^{-1}$-separated set on $\Sigma$,  and $\{y_i\}$ be a maximal $\la^{-1}2^\ell$-separated set on $M$. In local coordinate, let $\alpha\in C_0^\infty$ with $\alpha_j=\alpha(\la(x-x_j))$ being a partition of unity on $\Sigma$, and similarly let $ \beta\in C_0^\infty$ with $\beta_i^\ell= \beta(\la2^{-\ell}(y-y_i))$ being a partition of unity on $M$.  For each fixed $j$, there are only finitely many $i$ such that the support of $\alpha_j$ and $\beta_i^\ell$ 
are within distance $\la^{-1}2^\ell$. Since the number of such $i$ is bounded by some constant independent of $\ell$, for simplicity we may assume that there is just one such $i_j$ with $|y_{i_j}-x_j|\ls \la^{-1}2^\ell$.

We split the kernel $m_1(P^0,\la-1)(x,y)$ into the sum of 
\[K_0(x,y)=m_1(P^0,\la-1)(x,y)\1_{d_g(x,y)<\la^{-1}}\]
and
$$K_\ell(x,y) =m_1(P^0,\la-1)(x,y)\1_{d_g(x,y)\approx  \la^{-1}2^\ell}$$ with $1\le 2^\ell\ll \la$. The operators associated with these kernels are denoted by $K_0$ and $K_\ell$ respectively.

Since $V\in \mathcal{K}(M)$, for any $\eps_1>0$, we have for large $\la$
\[\sup_{x\in  M}\int_{d_g(y,x)<\la^{-1}} K_0(x,y)|V(y)||g(y)|dy<\eps_1\|\beta^0_{i_j} g\|_{L^\infty(M)}.\]
Since $V\in L^{n/2}(M)$, we can split $V$ into the sum of a bounded part and an unbounded part with small $L^{n/2}$-norm. So we may only consider the unbounded part and assume  $\|V\|_{L^{n/2}(M)}<\eps_1$, while the bounded part can be handled similarly with a better bound. Then
\[\sup_{x\in  M}\int_{d(y,x)\approx \la^{-1}2^\ell} K_\ell(x,y)|V(y)||g(y)|dy<\eps_1 2^{-N\ell}\|\beta_{i_j}^\ell g\|_{L^\infty(M)}.\]
Let $\rho\in \mathcal{S}(\mathbb{R})$ be nonnegative and satisfy $\rho(0)=1$ and $\hat \rho(t)=0$ if $|t|>1$. Let \[\eta_\ell( \la^{-1}2^\ell \tau)=\sum_\pm\rho(\la^{-1}2^\ell(\la\pm \tau))=\frac1\pi\int_{\mathbb{R}} \la 2^{-\ell}\hat \rho( \la 2^{-\ell} t)e^{it\la}\cos t\tau dt.\]
So 
$\eta_\ell(\la^{-1}2^\ell\tau)\approx 1$ when $|\tau-\la|\le 1$. 
By the finite propagation property \cite[Theorems 3.3 \& 3.4]{cosi}, the wave kernel $\cos tP_V(x,y)$ vanishes if $d_g(x,y)>t$. So we have $$\eta_\ell( \la^{-1}2^\ell P_V)(x,y)=0\ \ \text{when}\  \ d_g(x,y)>\la^{-1}2^{\ell}.$$  Since 
$\rho\in \mathcal{S}(\mathbb{R})$, we have for $\tau>0$
\[|\eta_\ell(\la^{-1}2^\ell\tau)|\ls (1+\la^{-1}2^\ell|\tau-\la|)^{-N},\ \forall N.\]
Then by the eigenfunction bounds \eqref{BSSLp} we get $$\sup_{x,y\in M}|\eta_\ell(\la^{-1}2^\ell P_V)(x,y)|\ls \la^{n}2^{-\ell}.$$ Thus, by Young's inequality, we have
\[\|\eta_\ell( \la^{-1}2^\ell P_V)f\|_{L^\infty(M)}\ls \la^{n/2}(2^\ell)^{-1+\frac n{2}}\|f\|_{L^{2}(M)}.\]
Then for each fixed $j$ and $\ell$, we have 
\begin{align*}
	&\|\alpha_jK_\ell(V\cdot \1_\la(P_V)f)\|^{2}_{L^{2}(\Sigma)}\\
	&\ls \la^{-k}\|\alpha_jK_\ell(V\cdot \1_\la(P_V)f)\|^{2}_{L^{\infty}(\Sigma)}\\
	&\ls  \la^{-k}\cdot \eps_1^22^{-2N\ell}\|\beta_{i_j}^\ell\cdot \1_\la(P_V)f\|^{2}_{L^{\infty}(M)}\\
	&= \la^{-k}\cdot \eps_1^22^{-2N\ell}\|\beta_{i_j}^\ell\cdot \eta_\ell(\la^{-1}2^\ell P_V)\eta_\ell(\la^{-1}2^\ell P_V)^{-1}\1_\la(P_V)f\|^2_{L^{\infty}(M)}\\
	&\ls \la^{-k}\cdot \eps_1^2 2^{-2N\ell}\cdot \la^n\|\tilde \beta_{i_j}^\ell\cdot \eta_\ell(\la^{-1}2^\ell P_V)^{-1}\1_\la(P_V)f\|^{2}_{L^{2}(M)}.
\end{align*}
Here the function $\tilde \beta_{i_j}^\ell$ is supported in  the $4\la^{-1}2^\ell$-neighborhood of $y_{i_j}$. Thus, we have 
\begin{equation}\label{eqboots}
	\begin{aligned}
		&\|K_\ell(V\cdot \1_\la(P_V)f)\|_{L^{2}(\Sigma)}\\
		&=(\sum_j\|\alpha_jK_\ell(V\cdot \1_\la(P_V)f)\|_{L^{2}(\Sigma)}^{2})^{\frac1{2}}\\
		&\ls\eps_1 \la^{\frac {n-k}{2}}2^{-N\ell }(\sum_j\|\tilde \beta_{i_j}^\ell\cdot \eta(\la^{-1}2^\ell P_V)^{-1}\1_\la(P_V)f\|^{2}_{L^{2}(M)})^\frac1{2}\\
		&\ls\eps_1 \la^{\frac {n-k}{2}}2^{-N\ell } \|\eta(\la^{-1}2^\ell P_V)^{-1}\1_\la(P_V)f\|_{L^{2}(\mathcal{T}_{\la^{-1}2^\ell}(\Sigma))}
	\end{aligned}
\end{equation}
Here $\mathcal{T}_r(\Sigma)$ is the $r$-neighborhood of $\Sigma$. 

To explain the idea to handle the  last term in \eqref{eqboots}, we first consider the special case $V \equiv 0$. We may assume in local coordinates
$$
\Sigma=\{(x, 0)\in \mathbb{R}^k \times \mathbb{R}^{n-k}: \ |x| \leq 1\} ,
$$
and then $\mathcal{T}_r(\Sigma)$ can be covered by
$$
\bigcup_{|y| \leq C_1r} \Sigma_{y}
$$
where for each $y \in \mathbb{R}^{n-k}$ we define
$$ 
\Sigma_{y}=\{(x,y)\in \mathbb{R}^k\times \mathbb{R}^{n-k}: |x| \le 2\} \text {. }
$$
By the proof of \eqref{bgt} in \cite{BGT}, the constant $C$ in \eqref{bgt} is uniform under small smooth perturbations on $\Sigma$, so  there exist constants $C_2>0$ and $\delta>0$ such that
\begin{equation}\label{pert}
	\sup_{|y|<\delta}\|\1_\la(P^0)\|_{L^2(M)\to L^{2}(\Sigma_{y})} \leq C_2 A.
\end{equation}
Thus
\begin{equation}\label{bt0}
	\begin{aligned}
		&\|\1_\la(P^0)f\|_{L^{2}(\mathcal{T}_r(\Sigma))} \\
		& \lesssim (\int_{|y| \le C_1 r}\|\1_\la(P^0)f\|^{2}_{L^{2}(\Sigma_{y})} d y)^{\frac1{2}} \\
		&\lesssim r^{\frac{n-k}{2}} A \|f\|_{L^2(M)}.
	\end{aligned}
\end{equation}
So if $V\equiv 0$ the last term in \eqref{eqboots} can bounded by a constant times $2^{-N\ell}A\|f\|_{L^2(M)}$. Summing over $\ell$ gives the desired bound $A$.

Next, we consider  general $V$.  To use the  bootstrap argument, the difficulty here is that the last term in \eqref{eqboots} is  the norm over some neighborhood of $\Sigma$ rather than the norm over $\Sigma$. To get around this, the key idea is to construct a closed foliation with leaves homeomorphic to $\Sigma$. We shall use an induction argument with respect  to $k$.

In the following, we shall work in a fixed local coordinate chart $U\subset M$ containing $\Sigma$. Let
$D^d(r) = \{y\in\mathbb{R}^d:|y|\leq r\}.$ Without loss of generality, we always assume that in this local coordinate, $\Sigma\subset D^n(2)$ and $D^n(100n)\subset U$.

\noindent \textbf{Base Step.} We start with the base case $k=1$. Let $\Sigma$ be a curve on $M$, by choosing local coordinates, we may assume it is
$$
\Sigma = \{(x_1,0)\in \mathbb{R}\times \mathbb{R}^{n-1}:\tfrac{1}{2}\leq x_1 \leq \tfrac{3}{2}\},
$$
and let 
$$
D_\Sigma = \{x\in\mathbb{R}^n:\tfrac{1}{2}\leq |x| \leq \tfrac{3}{2}\}.
$$
For any $\theta\in S^{n-1}\subset\mathbb{R}^n$, let $\Sigma_\theta$  be the intersection of $D_\Sigma$ and the ray from the origin with direction $\theta$.  Then we have
$$
D_\Sigma=\bigcup_{\theta \in S^{n-1}} \Sigma_\theta \sim \Sigma\times S^{n-1}.
$$
We will not distinguish $\Sigma_\theta$ and $D_\Sigma$ between their pullbacks, since the metrics are comparable. Let $$B=\sup _{\theta\in S^{n-1}}\|\1_\lambda(P_V)\|_{L^2(M) \to L^2(\Sigma_\theta)}.$$ By the $L^\infty$ bound in \eqref{BSSLp}, we have 
$
 B \ls\lambda^{\frac{n-1}{2}}<+\infty.
$ We shall prove that $B\ls A$.
\begin{figure}[h]
	\centering
	\includegraphics[width=0.8\textwidth]{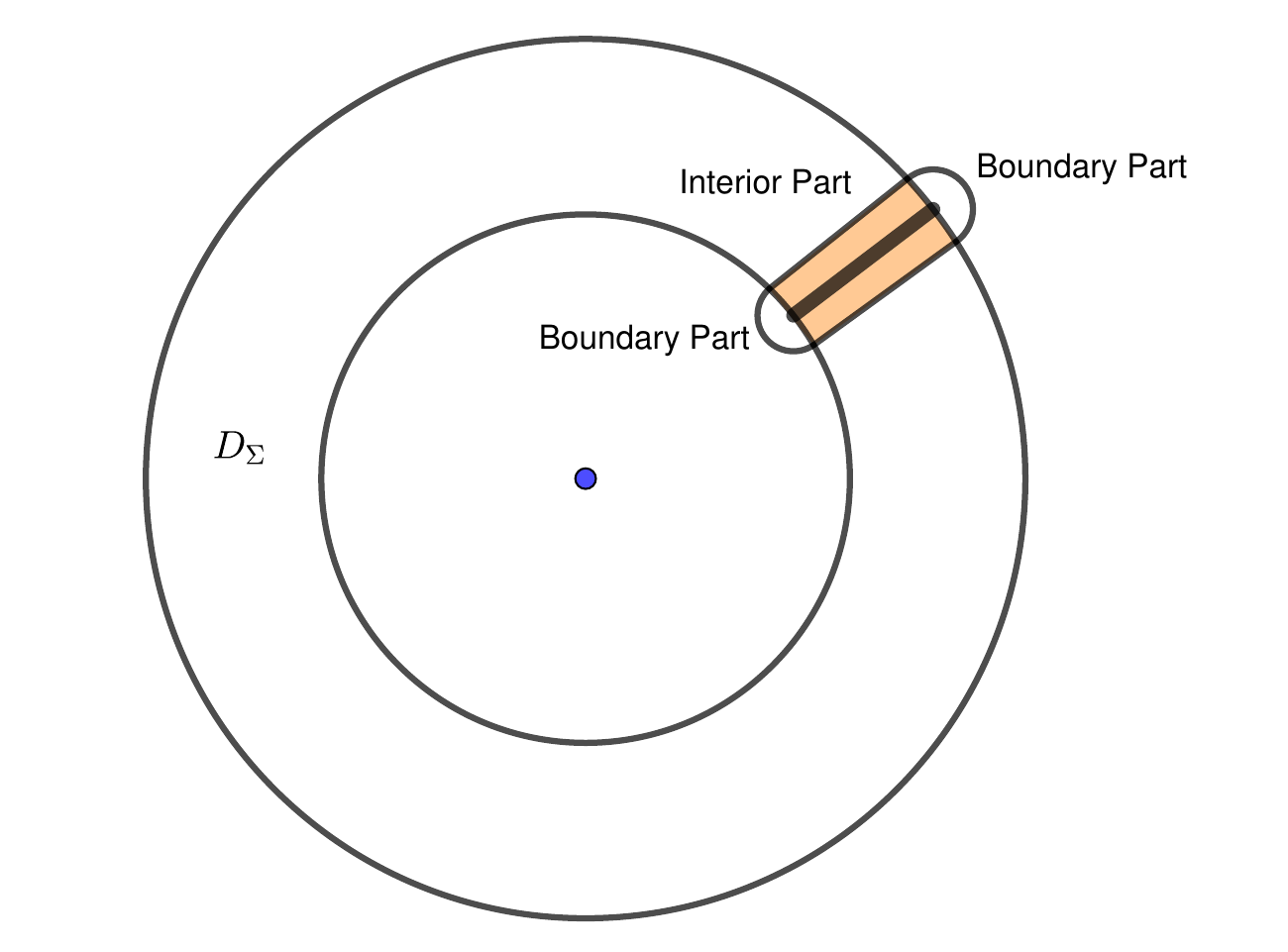}
	\caption{The neighborhood of $\Sigma_\theta$}
	\label{fig1}
\end{figure}

As in Figure \ref{fig1}, the neighborhood $\mathcal{T}_r(\Sigma_\theta)$ can be split into two parts.

\noindent \textbf{Interior Part:} $\mathcal{T}_r(\Sigma_\theta) \cap D_\Sigma$. As in \eqref{bt0}, using the above foliation, we have
\begin{equation}\label{part1}
	\begin{aligned}
		&\|\1_\la(P_V)f\|_{L^{2}(\mathcal{T}_r(\Sigma_\theta) \cap D_\Sigma)} \\
		& \lesssim (\int_{|\phi-\theta|\le r}\|\1_\la(P_V)f\|^{2}_{L^{2}(\Sigma_{\phi})} d \phi)^{\frac1{2}} \\
		&\lesssim r^{\frac{n-1}{2}} B \|f\|_{L^2(M)}.
	\end{aligned}
\end{equation}

\noindent \textbf{Boundary Part:} $\mathcal{T}_r(\Sigma_\theta) \setminus D_\Sigma$.
This part is essentially the $r$-neighborhood of $\partial \Sigma_\theta$, namely the endpoints of the curve $\Sigma_\theta$. Then by the $L^\infty$ bound in \eqref{BSSLp} we have 
\begin{equation}\label{part2}
	\begin{aligned}
		&\|\1_\la(P_V)f\|_{L^{2}(\mathcal{T}_r(\Sigma_\theta) \setminus D_\Sigma)} \\
		& \ls r^{\frac{n}2}\|\1_\la(P_V)f\|_{L^\infty(M)}\\
		&\ls r^{\frac{n}2}\la^{\frac{n-1}2}\|f\|_{L^2(M)}.
	\end{aligned}
\end{equation}
When $k=1$, $A=\la^{\frac{n-2}2}$ and $r=2^\ell \la^{-1}$, combining \eqref{part1}  with \eqref{part2} we have for any $\theta\in S^{n-1}$
$$
\begin{aligned}
	\lambda^{\frac{n-1}{2}}\|\1_\la(P_V) f\|_{L^2(\mathcal{T}_r(\Sigma_\theta))} &\lesssim  \lambda^{\frac{n-1}{2}}\cdot r^{\frac{n-1}{2}} B\|f\|_{L^2(M)}+\lambda^{\frac{n-1}{2}} \cdot \lambda^{\frac{n-1}{2}} r^{\frac{n}{2}}\|f\|_{L^2(M)} \\
	& = 2^{\frac{n-1}{2}\ell}B\|f\|_{L^2(M)}+2^{\frac{n}{2}\ell}A\|f\|_{L^2(M)}.
\end{aligned}
$$
Then we can estimate \eqref{eqboots} for $\Sigma_\theta$ by 
\begin{equation}\label{basecase}
	\|K_\ell(V\cdot \1_\la(P_V)f)\|_{L^2(\Sigma_\theta)}\ls (\eps_1 B+C_{\eps_1} A) 2^{-N\ell}\|f\|_{L^2(M)}.
\end{equation}
Summing over $\ell$, we  get the contribution $\eps_1 B +C_{\eps_1} A$ in this case. Combing this with the contributions $O(A)$ from Cases 1-4, we get $B\le \eps_1 B+C_{\eps_1} A$, which implies $B\ls A$ as desired.

\begin{remark}
	Note that in the argument for $k=1$ we only use the fact that $S^{n-1}$ is a closed manifold. So the same argument works for any smooth embedding $\Phi:\Sigma \times Q \rightarrow D^n(10)$, where $Q$ is smooth closed manifold of dimension $n-1$ and $\Phi(\Sigma, \theta_0)=\Sigma$ for some $\theta_0\in Q$.  This fact will be used to deal with $k\ge2$.
\end{remark}

 For $k\ge 2$, by rotation in local coordinates, one can similarly construct an embedding $\Phi:\Sigma\times S^{n-k} \rightarrow D^n(10)$. The Interior Part is similar, but the Boundary Part is more difficult to handle directly. To get around this difficulty, we generalize the family of embedded submanifolds and use an induction argument. Without loss of generality, we always assume $\Sigma$ is sufficiently small and homeomorhpic to a disk in $\mathbb{R}^k$. 

Let 
$$
\begin{aligned}
	\mathcal{F}(\Sigma) = \{&(Q,\Phi): Q \text{ is a smooth closed  manifold and }\dim Q =\dim M-\dim\Sigma, \\
	&\Phi:\Sigma \times Q \rightarrow D^n(10(n-k))\text{ is a smooth embedding}, \\
	&\Phi(\Sigma, \theta_0)=\Sigma \text{ for some } \theta_0 \in Q\}. \\
\end{aligned}
$$
From the above discussion, we see that $\mathcal{F}(\Sigma)$ is not empty. 

Let $\Sigma_\theta=\Phi(\Sigma,\theta)$.  We claim that for $n\geq 5$, $k\leq n-4$, $\dim\Sigma=k$, $(Q,\Phi)\in \mathcal{F}(\Sigma)$, there exists a constant $C = C(Q,\Phi,V,\Sigma,M)$ such that
\begin{equation}\label{induction}
	\sup _{\theta\in Q}\|\1_\lambda(P_V)\|_{L^2(M) \to L^2(\Sigma_\theta)} \leq CA(k).
\end{equation}
 In the following, we establish \eqref{induction} by an induction argument on $k$. The base case $k=1$ has been done as above. Suppose that \eqref{induction} holds for submanifolds of dimension $k-1$. Now we prove it for submanifolds of dimension $k$.

\noindent \textbf{Induction Step.} We fix any $(Q,\Phi)\in \mathcal{F}(\Sigma)$, and let $D_\Sigma = \Phi(\Sigma,Q)$. As before, we will not distinguish $\Sigma_x$ and $D_\Sigma$ between their pullbacks, since the metrics are comparable the ones on $Q\times\Sigma$. Let $$B = \sup _{\theta\in Q}\|\1_\lambda(P_V)\|_{L^2(M) \to L^2(\Sigma_\theta)}.$$ By the $L^\infty$ bound in \eqref{BSSLp}, we have 
$
B \ls\lambda^{\frac{n-1}{2}}<+\infty.
$ As before, the neighborhood $\mathcal{T}_r(\Sigma_\theta)$ can be split into two parts.

\noindent \textbf{Interior Part:} $\mathcal{T}_r(\Sigma_\theta) \cap D_\Sigma$, this contribution of this part can be handled as in \eqref{part1} by $r^{\frac{n-k}2}B\|f\|_{L^2(M)}$.

\noindent \textbf{Boundary Part:} $\mathcal{T}_r(\Sigma_\theta) \setminus D_\Sigma$ is essentially contained in $\mathcal{T}_{r}(\partial\Sigma_\theta)$. We need to do some extensions in order to apply the induction hypothesis to $\partial \Sigma_\theta$, which is a smooth submanifold of dimension $k-1$.
\begin{lemma}[Local extension lemma]\label{extlemma}
	Let $Z \subset \mathbb{R}^n$ be a submanifold of dimension $k-1$.
	Suppose that we have an smooth embedding $\Phi:Z\times D^{n-k}(1) \rightarrow D^n(10(n-k))$ and $\Phi(Z,0) = Z$. Then for any $x_0 \in Z$, there exist a neighborhood $U_0$ of $x_0$ in $Z$, and $\delta > 0$, a smooth closed manifold $\tilde Q$ of dimension $n-k+1$, and an embedding $\phi:D^{n-k}(\delta)\to \tilde Q$, an embedding $\tilde\Phi:U_0\times \tilde Q\to D^n(10(n-k+1))$ that extends the region $\Phi(U_0,D^{n-k}(\delta))$ in the following sense
	$$\tilde\Phi(x,\phi(y)) = \Phi(x,y)\ \ \text{ for }\ x\in U_0, \ y\in D^{n-k}(\delta).$$
\end{lemma}
We postpone the proof of this lemma and use it to finish the induction argument. We first split $\partial\Sigma_\theta$ into finitely many small enough pieces. For each piece $Z$, by Lemma \ref{extlemma} we can find a neighborhood $U_1$ of $\theta$, such that we can extend $\Phi(Z,U_1)$ to $\tilde\Phi(Z,\tilde Q)$ for some $(\tilde Q,\tilde\Phi)\in \mathcal{F}(Z)$. Then by induction hypothesis, we have  
\begin{align*}
	\|\1_\lambda(P_V)f\|_{L^2(\mathcal{T}_r(\partial\Sigma_\theta))} &\ls A(k-1)\cdot r^{\frac{n-(k-1)}{2}}  \|f\|_{L^2(M)}\\
	& =A(k)\cdot \lambda^{-\frac{n-k}{2}}\cdot 2^{\frac{n-k+1}{2}\ell}\|f\|_{L^2(M)}.
\end{align*}
By the compactness of $Q$, we can choose a constant uniform in $\theta\in Q$. Thus for $r=2^\ell\la^{-1}$, we have
$$
	\lambda^{\frac{n-k}{2}}\|\1_\la(P_V)f\|_{L^{2}(\mathcal{T}_r(\Sigma_\theta))}\ls 2^{\frac{n-k}{2}\ell}B\|f\|_{L^2(M)}+ A(k)\cdot 2^{\frac{n-k+1}{2}\ell}\|f\|_{L^2(M)}.
$$
Note that $N$ can be arbitrarily large in \eqref{eqboots}. Summing over $\ell$, we get the contribution  $\eps B+C_\eps A$ in this case. Finally, combing this with the contributions $O(A)$ from  Cases 1-4, we have $B\le \eps_1 B+C_{\eps_1} A$, which implies $B\ls A$ as desired.

\begin{proof}[Proof of Lemma \ref{extlemma}]
	
	Let $x=(x_1,\cdots,x_{k-1})$ be a local coordinate near $x_0$ on $Z$ and assume $x_0=0$. Let
	$\Phi_1(x,w)=s\nu+\Phi(x,y)$, where $w=(y,s)$ and $\nu\in \R^n$ is a unit normal vector of $\Phi(Z,D^{n-k}(1))$ at $\Phi(x_0,0)$.  Then we can find $\delta_1>0$ and a neighborhood $U_0$ of $x_0$ such that 
	$$\Phi_1:U_0\times D^{n-k+1}(\delta_1)\to D^n(10(n-k))$$ is a smooth embedding. 
	
	Let $e_i=\partial_{x_i}\Phi_1(0,0)$ for $1\leq i\leq k-1$, and $f_j=\partial_{w_j}\Phi_1(0,0)$ for $ 1\leq j\leq n-k+1$. By a linear transformation, we may assume $\{e_i\}\cup \{f_j\}$ is an orthonormal basis of $\R^n$. We extend $\Phi_1(0,D^{n-k+1}(\delta_1))$ to a smooth closed  manifold $\tilde Q\subset D^{n-k+1}(10\delta_1)$ that is homeomorhpic to the sphere $S^{n-k+1}$. Fix a sufficiently small $\delta_2>0$ and let $\mathcal{T}_{\delta_2}(\tilde Q)\subset \R^n$ be the $\delta_2$-neighborhood of $\tilde Q$. Then using the normal vector fields of $\tilde Q$, one can find a smooth bijection
	$$\Phi_2:D^{k-1}(\delta_2)\times \tilde Q \to \mathcal{T}_{\delta_2}(\tilde Q)$$ such that for each $q\in \tilde Q$, $\Phi_2(\cdot,q)$ is an isometry, and $\Phi_2(D^{k-1}(\delta_2),q)$ is a disk centered at $q$ and lies in the normal plane of $\tilde{Q}$ at $q$. We can choose the above $w\in D^{n-k+1}(\delta_1)$ as a local coordinate over $\Phi_1(0,D^{n-k+1}(\delta_1))\subset \tilde{Q}$. We will not distinguish the local coordinate and the corresponding point in the later calculations.
	
	Let $\bar{e}_i=\partial_{x_i}\Phi_2(0,0)$ for $1\leq i\leq k-1$, and $\bar{f}_j=\partial_{w_j}\Phi_2(0,0)$ for $ 1\leq j\leq n-k+1$. We have $\bar{f}_j=f_j$ since $\Phi_1(0,w)=\Phi_2(0,w)$ for $w$ near 0. Since both $\{e_i\}$ and $\{\bar{e}_i\}$ form an orthonormal basis of the normal plane of $\tilde{Q}$ at $w=0$, by applying a suitable orthogonal transformation, we may assume $\bar{e}_i=e_i$. By continuity, for any $\epsilon>0$, we can choose $\delta$ small enough such that when
	$|x|< \delta,\ |w|< \delta$, we have 
	\begin{equation}\label{gradest}
		|\partial_{x_i}\Phi_\alpha(x,w)-e_i| < \eps\ \ \text{and}\ \  |\partial_{w_j}\Phi_\alpha(x,w)-f_j| <\eps
	\end{equation}
	for $\alpha=1,2$ and $1\leq i\leq k-1$ and $1\leq j\leq n-k+1.$
	Choose a cutoff function $\eta\in C_0^{\infty} (D^{n-k+1}(\delta))$ with $\eta(w) = 1$ in $D^{n-k+1}(\delta/2)$, and let
	$$\tilde\Phi(x,w)  = \eta(w)\Phi_1(x,w) + (1-\eta(w))\Phi_2(x,w).$$
	Then for $|x|< \delta,\ |w|< \delta$, $\tilde{\Phi}(x,w)$ satisfies the same derivative estimates as in \eqref{gradest}, and so its Jacobian has full rank when $\eps$ is sufficiently small. So $$\tilde\Phi:D^{k-1}(\delta)\times \tilde Q \to D^n(10(n-k+1))$$ is the desired embedding.
\end{proof}

\bigskip

\section{Refinement on Case 2}\label{refine}
  In this section, we aim to refine the argument in Case 2 and remove the log loss there. Note that the resolvent-like symbol  $(\la_j^2-\tau_k^2)^{-1}$ naturally  appears in the previous perturbative argument for the wave kernels. To remove the log loss, we shall use the kernel decomposition of the resolvent operator $(-\Delta_g-(\la+i\eps )^2)^{-1}$ as in \cite{BSSY} and \cite{shaoyao}. As before, we assume $\eps=1$ or $(\log\la)^{-1}$.

\noindent \textbf{Case 2.}  $|\tau_k-\la|\le\eps$, $|\la_j-\la|\in(2^\ell,2^{\ell+1}]$, $\eps\le 2^\ell\le \la$.

Recall that we split the frequencies by the cutoff function $\psi\in C_0^\infty(\mathbb{R})$ satisfying $\psi(t)=1$ if $|t|\le  2$ and $\psi(t)=0$ if  $|t|>3$.
In this case,  $m(\la_j,\tau_k)=\frac{-\chi_\la(\tau_k)}{\la_j^2-\tau_k^2}\psi(\la_j/\la)$. Let $$m_1(\la_j,\tau_k)=\frac{-\chi_\la(\tau_k)}{\la_j^2-\tau_k^2+i\la\eps}\psi(\la_j/\la)$$ and $$m_2(\la_j,\tau_k)=m(\la_j,\tau_k)-m_1(\la_j,\tau_k).$$
We first handle $m_2$. It is clear that 
\begin{equation}
m_2(\la_j,\tau_k)=\frac{-i\la\eps \chi_\la(\tau_k)}{(\la_j^2-\tau_k^2)(\la_j^2-\tau_k^2+i\la\eps)}\psi(\la_j/\la).
\end{equation}
 For $|s-\la|\le \eps$ and $|\la_j-\la|\in(2^\ell,2^{\ell+1}]$,
we have
\[|m_2(\la_j,s)|+\eps|\partial_s m_2(\la_j,s)|\ls \eps\la^{-1}2^{-2\ell}.\]
We can use the same argument as Case 1 to handle 
\begin{align*}&\sum_{|\la_j-\la|\in(2^\ell,2^{\ell+1}]}\sum_{|\tau_k-\la|\le\eps}\int_Mm_2(\la_j,\tau_k)e_j^0(x)e_j^0(z)e_{\tau_k}(z)e_{\tau_k}(y)V(z)dz\\
	&= \sum_{|\la_j-\la|\in(2^\ell,2^{\ell+1}]}\sum_{|\tau_k-\la|\le\eps}\int_M\int_{\la-\eps}^{\la+\eps}\partial_sm_2(\la_j,s)\1_{[\la-\eps,\tau_k]}(s)e_j^0(x)e_j^0(z)e_{\tau_k}(z)e_{\tau_k}(y)V(z)dzds\\
&\ \ \ +\sum_{|\la_j-\la|\in(2^\ell,2^{\ell+1}]}\sum_{|\tau_k-\la|\le\eps}\int_Mm_2(\la_j,\la-\eps)e_j^0(x)e_j^0(z)e_{\tau_k}(z)e_{\tau_k}(y)V(z)dz\\
&=K_{1,\ell}(x,y)+K_{2,\ell}(x,y).\end{align*}

We handle $K_{2,\ell}$ first. For any $f\in L^{2}(M)$, we have
\begin{align*}
	\|K_{2,\ell}f\|_{X}&=\|\1_{\la,\ell}(P^0)m(P^0,\la-\eps)(V\cdot \1_\la(P_V)f)\|_{X}\\
	&\ls A(2^\ell/\eps)^{1/2}\|\1_{\la,\ell}(P^0)m(P^0,\la-\eps)(V\cdot \1_\la(P_V)f)\|_{L^2}\\
	&\ls A(2^\ell/\eps)^{1/2}\eps\la^{-1}2^{-2\ell}\|\1_{\la,\ell}(P^0)(V\cdot \1_\la(P_V)f)\|_{L^2}\\
	&\ls A(2^\ell/\eps)^{1/2}\eps\la^{-1}2^{-2\ell}\la^{\sigma(p_0)}2^{\ell/2}\|V\cdot\1_\la(P_V)f\|_{L^{p_0'}}\\
	&\ls A(2^\ell/\eps)^{1/2}\eps\la^{-1}2^{-2\ell}\la^{\sigma(p_0)}2^{\ell/2}\|V\|_{L^{q}}\|\1_\la(P_V)f\|_{L^{q_0}}\\
	&\ls A(2^\ell/\eps)^{1/2}\eps\la^{-1}2^{-2\ell}\la^{\sigma(p_0)}2^{\ell/2}\la^{\sigma(q_0)}\eps^{1/2}\|V\|_{L^{q}}\|f\|_{L^{2}}\\
	&=A2^{-\ell}\eps\|V\|_{L^{q}}\|f\|_{L^{2}}
\end{align*}
   
The method to handle $K_{1,\ell}$ is similar. Note that there is no log loss if we sum over $\ell$.

Next, we handle $m_1$.   It suffices to deal with $m_1(\la_j,\tau_k)$ for all $\la_j\le 3\la$, as the easy case $|\la_j-\la|\le\eps$ can be handled similarly as in Case 1. 

We write
\begin{align*}&\sum_{\la_j}\sum_{|\tau_k-\la|\le\eps}\int_Mm_1(\la_j,\tau_k)e_j^0(x)e_j^0(z)e_{\tau_k}(z)e_{\tau_k}(y)V(z)dz\\
	&= \sum_{\la_j}\sum_{|\tau_k-\la|\le\eps}\int_M\int_{\la-\eps}^{\la+\eps}\partial_sm_1(\la_j,s)\1_{[\la-\eps,\tau_k]}(s)e_j^0(x)e_j^0(z)e_{\tau_k}(z)e_{\tau_k}(y)V(z)dzds\\
	&\ \ \ +\sum_{\la_j}\sum_{|\tau_k-\la|\le\eps}\int_Mm_1(\la_j,\la-\eps)e_j^0(x)e_j^0(z)e_{\tau_k}(z)e_{\tau_k}(y)V(z)dz\\
	&=K_{1}(x,y)+K_{2}(x,y).\end{align*}
For $|s-\la|\le \eps$,
we have
\[|m_1(\la_j,s)|+\eps|\partial_s m_1(\la_j,s)|\ls \la^{-1}.\]
As before, we just handle $K_2$. It suffices to prove 
\begin{equation}\label{lowres}
		\|(-\Delta_g-(\la+i\eps)^2)^{-1}\psi(P^0/\la)\|_{L^{p_0'}(M)\to X}\ls A_1\la^{\sigma(p_0)-1},
\end{equation}
where $A_1$ is defined to be the value of $A$ at   $\eps=1$. It satisfies $A_1\eps^{1/2}\ls A\ls A_1$.
Indeed, by \eqref{lowres} we have
\begin{align*}
	\|m_1(P^0,\la-\eps)(V\cdot \1_\la(P_V)f)\|_{X}&\ls A_1\la^{\sigma(p_0)-1}\|V\cdot \1_\la(P_V)f\|_{L^{p_0'}}\\
	& \ls A_1\la^{\sigma(p_0)-1}\|V\|_{L^q}\|\1_\la(P_V)f\|_{L^{q_0}}\\
	&\ls A_1\la^{\sigma(p_0)-1}\la^{\sigma(q_0)}\eps^{1/2}\|V\|_{L^q}\|f\|_{L^2}\\
	&=A_1\eps^{1/2}\|V\|_{L^q}\|f\|_{L^2}.
\end{align*}
This gives us the desired bound.

In the following, we aim to prove \eqref{lowres}, albeit with a potential loss in certain cases. As in \cite[Section 3]{BHSS}, we shall use the formula
\[(-\Delta_g-(\la+i\eps)^2)^{-1}=\frac i{\la+i\eps}\int_0^\infty e^{i\la t}e^{-\eps t}\cos tP^0dt.\]
Let $\rho\in C_0^\infty(\mathbb{R})$ satisfy \[\1_{[-\eps_0/2,\eps_0/2]}\le \rho\le \1_{[-\eps_0,\eps_0]},\]
where $\eps_0=\min\{1,\text{Inj}(M)/2\}$ with Inj($M$) denoting the injectivity radius of $(M, g)$. Let $\beta\in C_0^\infty(\mathbb{R})$ satisfy
\[|\beta(t)|\le1,\ \supp\beta\subset [1/2,2],\ \ \sum_{j\in\mathbb{Z}}\beta(2^{-j}t)=1,\ t>0.\]Denote
\[\tilde\rho(t)=1-\sum_{j\ge0}\beta(2^{-j}t),\ t>0.\]Then $\supp\tilde\rho\subset [-4,4]$.
Let 
\[R_\la(\tau)=\frac i{\la+i\eps}\int_0^\infty (1-\rho(\eps t))e^{i\la t}e^{-\eps t}\cos t\tau dt\]
\[S_0(\tau)=\frac i{\la+i\eps}\int_0^\infty \tilde\rho(\la t)\rho(\eps t)e^{i\la t}e^{-\eps t}\cos t\tau dt\]
and for $j=1,2,...,[\log_2(\la/\eps)]$
\[S_j(\tau)=\frac i{\la+i\eps}\int_0^\infty \beta(\la2^{-j}t)\rho(\eps t)e^{i\la t}e^{-\eps t}\cos t\tau dt.\]
Then we write
\begin{equation}\label{resker}
	(-\Delta_g-(\la+i\eps)^2)^{-1}\psi(P^0/\la)=\Big(S_0(P^0)+\sum_{1< 2^j\le \la}S_j(P^0)+\sum_{\la< 2^j\le \la/\eps}S_j(P^0)+R_\la(P^0)\Big)\psi(P^0/\la).
\end{equation}
Here for $\tau\ge0$ we have
\begin{equation}\label{Rla}|R_\la(\tau)\psi(\tau/\la)|\ls \la^{-1}\eps^{-1}(1+\eps^{-1}|\la-\tau|)^{-N},\ \ \forall N,\end{equation}
and for  $j=0,1,2,...,[\log_2(\la/\eps)]$  
\begin{equation}\label{Sj}|S_j(\tau)\psi(\tau/\la)|\ls \la^{-2}2^j(1+\la^{-1}2^j|\la-\tau|)^{-N},\ \ \forall N.\end{equation}

Then by the spectral theorem with a simple duality argument (see e.g. \cite[Proof of Lemma 2.3]{BSSY}) we obtain
\[\|R_\la(P^0)\psi(P^0/\la)\|_{L^{p_0'}(M)\to X}\ls A_1\la^{\sigma(p_0)-1},\]
and for $j=0,1,2,...,$
\[\|S_j(P^0)\psi(P^0/\la)\|_{L^{p_0'}(M)\to X}\ls A_1\la^{\sigma(p_0)-1}.\] 
So we get
\begin{equation}\label{nonlocal1}
	\|\sum_{\la< 2^j\le \la/\eps}S_j(P^0)\psi(P^0/\la)\|_{L^{p_0'}(M)\to X}\ls A_1\la^{\sigma(p_0)-1}\log(2+\eps^{-1})
\end{equation}
If $M$ is negatively curved and $\eps=(\log\la)^{-1}$, we obtain the kernel bound by \cite[(3.22)]{BHSS}
\begin{equation}\label{nonlocal2}
	\Big|\sum_{\la< 2^j\le \la/\eps}S_j(P^0)\psi(P^0/\la)(x,y)\Big|\ls_{\delta_0} \la^{\frac{n-3}2+\delta_0},\ \forall \delta_0>0.
\end{equation}
This trivially implies an $L^{p_0'}(M)\to X$ operator bound by Young's inequality. In some cases, it can be used to improve \eqref{nonlocal1}.

Let $\theta=\la^{-1}2^j$. For $1< 2^j\le \la$, by the proof of the kernel estimate of $S_j(P^0)$ in   \cite[(2.23)]{shaoyao}, we can obtain the kernel of its smooth cutoff in essentially  the same form
\begin{equation}\label{Sjker}S_j(P^0)\psi(P^0/\la)(x,y)=\la^{\frac{n-3}2}\theta^{-\frac{n-1}2}e^{i\la d_g(x,y)}a_\theta(x,y)+O(\la^{-1}\theta^{1-n}\1_{d_g(x,y)<4\theta})\end{equation}
where  the smooth function $a_\theta(x,y)$ is supported in $\{d_g(x,y)/\theta\in (1/4,4)\}$ and 
\[|\partial_{x,y}^\gamma a_\theta(x,y)|\le C_\gamma \theta^{-|\gamma|},\ \forall \gamma.\]
We denote the first term in \eqref{Sjker} by $T_j(x,y)$ and the remainder term by $r_j(x,y)$.

\subsection{Period integrals}
In this subsection, we prove \eqref{lowres} for $\|f\|_{X}=|\int_\Sigma fd\sigma|$ and $A_1=\la^{\frac{n-k-1}2}$, albeit with a potential $\log\log\la$ loss in the case $\eps=(\log\la)^{-1}$. We first handle the remainder term in \eqref{Sjker}. Indeed,
\begin{align*}
	\|r_jf\|_{X}&\ls \la^{-1}\theta^{1-n}\cdot \theta^k\|f\|_{L^1(\mathcal{T}_\theta(\Sigma))}\\
	&\ls \la^{-1}\theta^{1-n}\cdot \theta^{k}\theta^{\frac{n-k}{p_0}}\|f\|_{L^{p_0'}(M)}\\
	&\ls \la^{-1}\theta^{1-\frac{n-k}{p_0'}}\|f\|_{L^{p_0'}(M)}.
\end{align*}
Then 
\begin{align*}\sum_j\|r_jf\|_{X}&\ls( \la^{-1}+\la^{-2+\frac{n-k}{p_0'}})\|f\|_{L^{p_0'}(M)}.
\end{align*}
It is better than the desired bound, since 
\[-2+\frac{n-k}{p_0'}<-2+\frac n{p_0'}-\frac k2=\frac{n-k-1}2+\sigma(p_0)-1.\]
For the following type oscillatory integrals
$$
\int_\Sigma e^{i\la d_g(x,y)}a(x,y)dx,
$$
where $a\in C_0^\infty$ is supported in $\{(x,y): d_g(x,y)\approx 1\}$ and $|\partial_{x,y}^\gamma a(x,y)|\le C_\gamma.$ If the support of $a(x,y)$ is sufficiently small, then $x$ is a critical point for the phase function
if and only if the geodesic connecting $x,y$ is perpendicular to $\Sigma$, and in this case the hessian is nondegenerate. Thus, by stationary phase we know it is of size $O(\lambda^{-k/2})$.

Let $\{x_\ell\}$ be a maximal $\theta$-separated set in $\Sigma$. Suppose that $\rho_\theta$ is a smooth cutoff function on $\Sigma$ with support of diameter $\sim \theta$, and $\eta_\theta$ is a smooth cutoff function on $M$ with support of diameter $\sim \theta$, and $\{\rho_\theta(x-x_\ell)\}$ is a partition of unity on $\Sigma$. Then by rescaling and stationary phase we have
\begin{align*}
	\|T_jf\|_{X}&\ls \la^{\frac{n-3}2}\theta^{-\frac{n-1}2}\sum_\ell \Big|\iint e^{i\la d_g(x,y)}a_\theta(x,y)\rho_\theta(x-x_\ell)\eta_\theta(y-x_\ell)f(y)dydx\Big|\\
	&= \la^{\frac{n-3}2}\theta^{-\frac{n-1}2}\theta^{n+k}\sum_\ell \Big|\iint e^{i2^j d_g(\theta x+x_\ell,\theta y+x_\ell)/\theta}a_\theta(\theta x+x_\ell,\theta y+x_\ell)\rho_\theta(\theta x)\eta_\theta(\theta y)f(\theta y+x_\ell)dydx\Big|\\
	&\ls \la^{\frac{n-3}2}\theta^{-\frac{n-1}2}\theta^{n+k}(2^j)^{-k/2}\sum_\ell \int|\eta_\theta(\theta y)f(\theta y+x_\ell)|dy\\
	&\ls \la^{\frac{n-3}2}\theta^{-\frac{n-1}2}\cdot  \theta^{n+k}\cdot (2^j)^{-k/2}\cdot \theta^{-n}\|f\|_{L^1(\mathcal{T}_\theta(\Sigma))}\\
	&\ls \la^{\frac{n-3}2}\theta^{-\frac{n-1}2}\cdot \theta^{k}\cdot (2^j)^{-k/2}\cdot \theta^{\frac{n-k}{p_0}}\|f\|_{L^{p_0'}(M)}\\
	&=\la^{\frac{n-k-3}2}\theta^{-\frac{n-k-1}2+\frac{n-k}{p_0}}\|f\|_{L^{p_0'}(M)}.
\end{align*}
Here we apply stationary phase to the phase function $d_\theta(x,y)=d_g(\theta x+x_\ell,\theta y+x_\ell)/\theta$ and use the fact that $d_\theta(x,y)$ becomes arbitrarily close to Euclidean distance in $C^\infty$-topology as $\theta\to 0$. Since $\la^{-1}\le \theta=\la^{-1}2^j\le 1$, we obtain
\begin{align*}
	\sum_j\|T_jf\|_{X}&\ls (\la^{\frac{n-k-3}2}\log\la+\la^{-2+\frac{n-k}{p_0'}}\log\la)\|f\|_{L^{p_0'}(M)}.
\end{align*}
It is better than the desired bound, since
\[\frac{n-k-3}2<\frac{n-k-1}2+\sigma(p_0)-1.\] So we complete the proof of \eqref{lowres} for period integrals, albeit with a potential $\log\log\la$ loss from  \eqref{nonlocal1} if  $\eps=(\log\la)^{-1}$. 

Since we have removed the log loss in Case 2, by the main argument in Section \ref{MG} we have proved Theorem \ref{mainthm2} for $k>n-4$ and Theorem \ref{pineg} without log loss. To remove the log loss for $k\le n/2$ and $n\ge5$, we  control the period integrals by the $L^2(\Sigma)$ norms, since $\delta(k,2)$ is exactly $\frac{n-k-1}2$ in this case. This  will be addressed in the next subsection.

\subsection{Restriction bounds} In this subsection, we prove \eqref{lowres} for $\|f\|_{X}=\|f\|_{L^{p_c}(\Sigma)}$ or $\|f\|_{L^2(\Sigma)}$.  Let $p_2=p$ if $\|f\|_X=\|f\|_{L^{p}(\Sigma)}$.


We first handle the remainder term in \eqref{Sjker}. By Young's inequality
\begin{align*}
	\|r_jf\|_{X}&\ls \la^{-1}\theta^{1-n}\cdot \theta^{\frac k{p_2}+\frac n{p_0}}\|f\|_{L^{p_0'}(M)}\\
&=\la^{n-2-\frac k{p_2}-\frac n{p_0}}2^{j(1-n+\frac k{p_2}+\frac n{p_0})}\|f\|_{L^{p_0'}(M)}
\end{align*}
Then
\[n-2-\frac k{p_2}-\frac n{p_0}\le\delta(k,p_2)+\sigma(p_0)-1,\]
\[1-n+\frac k{p_2}+\frac n{p_0}<0.\]
So summing over $j$ we get
\[\sum_{j}\|r_jf\|_{X}\ls \la^{\delta(k,p_2)+\sigma(p_0)-1}\|f\|_{L^{p_0'}(M)}.\]
Next, we handle $T_j$.  Let $T^\la$ be the oscillatory integral operator 
\begin{equation}\label{oscT}
T^\la f(x)=\int_M e^{i\la d_g(x,y)}a(x,y)f(y)dy,
\end{equation}
where $dy$ is the volume measure on $M$, and the smooth function $a(x,y)$ is supported in $\{d_g(x,y)\in (\frac14,4)\}$ and $|\partial_{x,y}^\gamma a(x,y)|\le C_\gamma.$
By the proof of the $L^2(M)-L^{p_c}(\Sigma)$ bound \eqref{bgt}  and interpolation with the trivial $L^1(M)-L^{p_c}(\Sigma)$ bound, we have for some $\delta_0>0$
\begin{equation}\label{key}\|T^\la \|_{L^{p_0'}(M)\to L^{p_c}(\Sigma)}\ls \la^{-\frac {2k}{p_0p_c}}\ls \la^{\frac{n-1}2-\frac n{p_0}-\frac k{p_c}-\delta_0},\end{equation}
whenever
\begin{equation}\label{krange}k<\frac{p_c}2\Big(n-\frac {1}{1-2/p_0}\Big).\end{equation}
We postpone the discussion on this condition and use it to obtain \eqref{lowres} first. 

Let $\{x_\ell\}$ be a maximal $\theta$-separated set in $\Sigma$. Suppose that $\rho_\theta$ is a smooth cutoff function on $\Sigma$ with support of diameter $\sim \theta$, and $\eta_\theta$ is a smooth cutoff function on $M$ with support of diameter $\sim \theta$, and $\{\rho_\theta(x-x_\ell)\}$ is a partition of unity on $\Sigma$. So by rescaling and using \eqref{key} we have
\begin{align*}
	&\|T_jf\|_{L^{p_c}(\Sigma)}= \la^{\frac{n-3}2}\theta^{-\frac{n-1}2}(\sum_\ell\int_\Sigma\Big|\int_M e^{i\la d_g(x,y)}a_\theta(x,y)\rho_\theta( x-x_\ell)\eta_\theta(y-x_\ell)f(y)dy\Big|^{p_c}dx)^{\frac1{p_c}}\\
	&=\la^{\frac{n-3}2}\theta^{-\frac{n-1}2}\theta^{n+\frac k{p_c}}(\sum_\ell\int_\Sigma\Big|\int_M e^{i2^j d_g(\theta x+x_\ell,\theta y+x_\ell)/\theta}a_\theta(\theta x+x_\ell,\theta y+x_\ell)\rho_\theta(\theta x)\eta_\theta(\theta y)f(\theta y+x_\ell)dy\Big|^{p_c}dx)^{\frac1{p_c}}\\
	&\ls\la^{\frac{n-3}2}\theta^{-\frac{n-1}2}\theta^{n+\frac k{p_c}}2^{j(\frac{n-1}2-\frac n{p_0}-\frac k{p_c}-\delta_0)}(\sum_\ell\|\eta_\theta(\theta y)f(\theta y+x_\ell)\|_{L^{p_0'}(M)}^{p_c})^{\frac1{p_c}}\\
	&\ls \la^{\frac{n-3}2}\theta^{-\frac{n-1}2}\theta^{n+\frac k{p_c}}2^{j(\frac{n-1}2-\frac n{p_0}-\frac k{p_c}-\delta_0)}(\sum_\ell\|\eta_\theta(\theta y)f(\theta y+x_\ell)\|_{L^{p_0'}(M)}^{p_0'})^{\frac1{p_0'}}\\
	&\ls \la^{\frac{n-3}2}\theta^{-\frac{n-1}2}\theta^{n+\frac k{p_c}}2^{j(\frac{n-1}2-\frac n{p_0}-\frac k{p_c}-\delta_0)}\theta^{-n/p_0'}\|f\|_{L^{p_0'}(M)}\\
	&=\la^{n-2-\frac k{p_c}-\frac n{p_0}}2^{-\delta_0j}\|f\|_{L^{p_0'}(M)}.
\end{align*}
Here we apply \eqref{key} to the operator with the phase $d_\theta(x,y)=d_g(\theta x+x_\ell,\theta y+x_\ell)/\theta$ and use the fact that $d_\theta(x,y)$ becomes arbitrarily close to Euclidean distance in $C^\infty$-topology as $\theta\to 0$. 
So summing over $j$ we get
\[\sum_{j}\|T_jf\|_{L^{p_c}(\Sigma)}\ls \la^{\delta(k,p_c)+\sigma(p_0)-1}\|f\|_{L^{p_0'}(M)}.\]
When $n=2$,  \eqref{krange} holds for $k=1$, $p_0=\infty$, $p_c=4$. When $n\ge3$, we fix $\frac1{p_0}=\frac{n+3}{2n+2}-\frac 2n$. Then  \eqref{krange} holds whenever 
\[k<\frac{p_c}2\Big(n-\frac1{1-2(\frac{n+3}{2n+2}-\frac 2n)}\Big)=\frac{p_c}2\cdot \frac{n(n+3)}{2(n+2)}.\]
When $k\le n-2$ and $p_c=2$, \eqref{krange} holds for $k\le n/2$. Similarly, when $k=n-1$ and $p_c=\frac{2n}{n-1}$, \eqref{krange} holds for $n=3,4.$  At the endpoint $p=2$ when $k=n-1$,  the log loss can be removed for $n=3,4$ and we postpone the proof to the end of this subsection.

When $\eps=1$, we obtain \eqref{lowres}  for $k\le n/2$ when $n\ge5$ and for all $k$ when $n=2,3,4$. For the remaining cases, we shall prove \eqref{lowres} for $p>q_k$ for some $q_k<2.3$ at the end of this subsection. These together complete the proof of Theorem \ref{mainthm}. 

To prove Theorem \ref{thmneg}, we need to use  the kernel bound \eqref{nonlocal2} to remove the $\log\log\la$ loss from \eqref{nonlocal1} when $\eps=(\log\la)^{-1}$. Indeed,  for $n\ge2$ we have
\[\frac{n-3}2<\delta(1,p_c)+\sigma(p_0)-1.\] 
So when $\eps=(\log\la)^{-1}$ we get \eqref{lowres} with $k=1$ for $n\ge2$. We remark that the argument in Subsection \ref{hcd} we only consider $\eps=1$, but the Base Step in Subsection \ref{hcd} can be used to obtain improved geodesic  restriction estimates with $\eps=(\log\la)^{-1}$, since a tubular neighborhood of a  geodesic segment can be viewed as the union of a family of geodesic segments. So we finish the proof of Theorem \ref{thmneg}.

\begin{remark}\label{rem3.1} Let $n\ge3$. To remove the log loss in Case 2, it suffices to establish
	\begin{equation}\label{beat}
		\|T^\la\|_{L^{p_0'}(M)\to L^{p_c}(\Sigma)}\ls \la^{\frac{n-1}2-\frac n{p_0}-\frac k{p_c}-\delta_0}
	\end{equation}	
	for some $\delta_0>0$ and $\frac1{p_0}\in [\frac{n+3}{2n+2}-\frac 2n,\frac{n-1}{2n+2}]$. We have proved it for $k\le n/2$ when $n\ge5$ and for all $k$ when $n\le4$. Moreover, in Section \ref{sectosc} we shall discuss it further in the model case where the metric is  Euclidean and the submanifold is flat. In this case, we can improve the range to  $k\le [2n/3]-2$ when $n\ge11$.
\end{remark}

 Since we have only handled the endpoint $p=p_c$ so far, to finish the proof we also need to remove the log loss for other exponents, including $p=2$ when $k=n-1$. 

 Let $n\ge3$. We first handle $p\ge p_c$. Suppose that for some $\alpha\ge0$
\begin{equation}\label{Lpkey}\|T^\la\|_{L^{p_0'}(M)\to L^p(\Sigma)}\ls \la^{-\alpha}.\end{equation}
Then by the previous argument we get
\begin{equation}\label{Tjb}\|T_j\|_{L^{p_0'}(M)\to L^p(\Sigma)}\ls \la^{n-2-\frac{k}p-\frac n{p_0}}2^{j(-\alpha-\frac{n-1}2+\frac n{p_0}+\frac kp)}.\end{equation}
To remove the log loss, we require that
\begin{equation}\label{cond1}n-2-\frac{k}p-\frac n{p_0}+\sigma(q_0)\le \frac{n-1}2-\frac kp
	\end{equation}
\begin{equation}\label{cond2}
	n-2-\frac{k}p-\frac n{p_0}+(-\alpha-\frac{n-1}2+\frac n{p_0}+\frac kp)+\sigma(q_0)<\frac{n-1}2-\frac kp.
\end{equation}
To get the largest range of $p$, we require that $q_0=\frac{2n+2}{n-1}$ and $\frac1{p_0}=\frac{n+3}{2n+2}-\frac2n$. Then \eqref{cond1} always holds, and \eqref{cond2} is equivalent to 
\begin{equation}\label{prange}
\alpha>	\frac kp-\frac{n+3}{2n+2}.
\end{equation}
Now it suffices to determine the best $\alpha$ in \eqref{Lpkey}.

When $k=n-1$,  by interpolation between $L^2-L^2$ and $L^1-L^\infty$ bounds we have
\begin{equation}\label{intp1}
	\|T^\la\|_{L^{p'}(M)\to L^p(\Sigma)}\ls \la^{-\frac1p(n-\frac32)}.\end{equation}
By Stein's oscillatory integral theorem (see \cite[Theorem 2.2.1]{fio}), we get
\begin{equation}\label{stein}\|T^\la\|_{L^{\frac{2n}{n+2}}(M)\to L^2(\Sigma)}\ls \la^{-\frac{(n-1)(n-2)}{2n}}.\end{equation}
By the interpolation between this and the $L^1-L^\infty$ bound, we get for $p_2=\frac{np}{n-2}$
\begin{equation}\label{intp2}
		\|T^\la\|_{L^{p_2'}(M)\to L^p(\Sigma)}\ls \la^{-(n-1)/p_2}.
\end{equation}
We require that $p<p_0<p_2$. By interpolation between \eqref{intp1} and \eqref{intp2} we get
\begin{equation}\label{bd1}
	\|T^\la\|_{L^{p_0'}(M)\to L^{p}(\Sigma)}\ls \la^{-\frac{n-2}{4p}-\frac{3n-4}{4p_0}}.
\end{equation}

When $k\le n-2$, similarly we have
\begin{equation}\label{intp11}
	\|T^\la\|_{L^{p'}(M)\to L^p(\Sigma)}\ls \la^{-k/p},\end{equation}
	and for $p_2=\frac{(k+1)p}{k-1}$
	\begin{equation}\label{intp22}
		\|T^\la\|_{L^{p_2'}(M)\to L^p(\Sigma)}\ls \la^{-k/p_2}.
	\end{equation}
	We require that $p<p_0<p_2$. By interpolation between \eqref{intp11} and \eqref{intp22} we get
	\begin{equation}\label{bd2}
		\|T^\la\|_{L^{p_0'}(M)\to L^{p}(\Sigma)}\ls \la^{-k/p_0}.
	\end{equation}
Thus, inserting the power $\alpha$ from \eqref{bd1} and \eqref{bd2} into \eqref{prange}, we can obtain the range of $p$. When $k=n-1$, we have the lower bound
\begin{equation}\label{r1}
	p>\frac{2n(-2+n+3n^2)}{16+4n-3n^2+3n^3}=2+\frac{8}{3n}-\frac4{3n^2}+O(\frac1{n^3})
\end{equation}
for $n\ge9$, and $p\ge p_c$ for $n\le 8$. 
When $k\le n-2$, we have the lower bounds
\begin{equation}\label{r2}
	p>\frac{2kn(n+1)}{k(n^2-n-4)+n^2+3n}=2+\frac{2(k+1)}{n^2}+O(\frac1{n^3}).
\end{equation}
It is easy to see that these lower bounds are less than 2.3.

Moreover, we can also consider $2\le p<p_c$ when $n\le 8$ and $k=n-1$. To remove the log loss, we require that
\begin{equation}\label{cond11}n-2-\frac{n-1}p-\frac n{p_0}+\sigma(q_0)\le \frac{n-1}4-\frac {n-2}{2p}
\end{equation}
\begin{equation}\label{cond21}
	n-2-\frac{n-1}p-\frac n{p_0}+(-\alpha-\frac{n-1}2+\frac n{p_0}+\frac {n-1}p)+\sigma(q_0)<\frac{n-1}4-\frac {n-2}{2p}.
\end{equation}
To get the largest range of $p$, we require that $q_0=\frac{2n+2}{n-1}$ and $\frac1{p_0}=\frac{n+3}{2n+2}-\frac2n$. Then \eqref{cond11} always holds, and \eqref{cond21} is equivalent to 
\begin{equation}\label{prange1}
\alpha>\frac{n-2}{2p}+\frac{n^2-2n-7}{4(n+1)}.
\end{equation}
Now it suffices to determine the best $\alpha$ in \eqref{Lpkey}. Indeed, we can interpolate between \eqref{intp2} and the trivial bound
\begin{equation}\label{young1}
	\|T^\la\|_{L^1(M)\to L^{p}(\Sigma)}\ls 1
	\end{equation}
to get
\begin{equation}\label{bd3}
		\|T^\la\|_{L^{p_0'}(M)\to L^{p}(\Sigma)}\ls \la^{-(n-1)/p_0}.
\end{equation}
Thus, inserting the power $\alpha$ from \eqref{bd3} (when $n\le 6$)  or \eqref{bd1} (when $n\ge7$) into \eqref{prange1} we get 
\[p>\begin{cases}
	6/5 = 1.200,\ \ \ \ \ \ \ \ n=3\\
	20/11\approx 1.818,\ \ \ \ \ n=4\\
	45/22\approx 2.045,\ \ \ \ \ n=5\\
	168/79\approx 2.127,\ \ \ \  n=6\\
	280/127\approx 2.205,\ \ \ n=7\\
	9/4=2.250,\ \ \ \ \ \ \ \ \ n=8.
\end{cases}\]
So in particular we can remove the loss at the endpoint $p=2$ when $n=3,4$. 
	\begin{remark}\label{compare}These ranges are larger than those by Blair-Park \cite[Theorems 1.3 \& 1.4]{BP}, since in the argument above we do not need the ``uniform  resolvent conditions'': $n-2-\frac{k}p-\frac n{p_0}=0$ and the power of $2^j$ is negative in \eqref{Tjb}. For instance, when $n\ge 8$ they proved that when $k=n-1$,
	\[p>\frac{2n^2-5n+4}{n^2-4n+8}=2+\frac3n+O\Big(\frac1{n^3}\Big),\]
	and when $k=n-2$,
	\[p>\frac{2(n-2)^2}{n^2-5n+8}=2+\frac2n+\frac2{n^2}+O\Big(\frac1{n^3}\Big).\]
	These can be compared with \eqref{r1} and \eqref{r2}.
\end{remark}

\section{Restriction to curved hypersurfaces}

In the potential-free case, for $k=n-1$ and $p\ge \frac{2n}{n-1}$, zonal functions saturate the restriction estimate in Theorem \ref{mainthm} for any hypersurface. However, for $ p<\frac{2n}{n-1}$, Burq--G\'erard--Tzvetkov \cite{BGT} and Hu \cite{Hu} proved that a power improvement is possible for curved hypersurfaces. We prove this type of improved estimates for singular potentials.

\begin{theorem}\label{mainthmcurv}Let $M$ be a compact manifold of dimension $n\ge2$. Let $\Sigma\subset M$  be a hypersurface with positive
	(or negative) definite second fundamental form. Suppose $2\le p<\frac{2n}{n-1}$.  If  $V\in \mathcal{K}(M)\cap L^{n/2}(M)$, then we have
	\begin{equation}\label{PVbdcurv}
		\|e_\lambda\|_{ L^p(\Sigma)}\ls \la^{\frac{n-1}3-\frac{2n-3}{3p}}\|e_\lambda\|_{L^2(M)},
	\end{equation}
albeit with a potential logarithmic loss  when $(n,p)\in  \{(n,p):n\ge 4, p <2.3\}$. The loss can be removed if $V\in L^q(M)$ with $q>n/2$.
\end{theorem}
As in Theorem \ref{thmneg}, we can obtain improved bounds on negatively curved surfaces. These results were established by Park \cite{park} for Laplacian eigenfunctions.

\begin{theorem}\label{thmcurv2}Let $M$ be a negatively curved surface. Let $\Sigma\subset M$  be a curve  with non-vanishing geodesic curvature. Suppose $2\le p<4$.  If  $V\in \mathcal{K}(M)$, then we have
	\begin{equation}\label{PVbdcurv2}
		\|e_\lambda\|_{ L^p(\Sigma)}\ls \la^{\frac{1}3-\frac{1}{3p}}(\log\la)^{-\frac12}\|e_\lambda\|_{L^2(M)}.
	\end{equation}
\end{theorem}
The proof of Theorem \ref{thmcurv2} is essentially the same as the proof of Theorem \ref{thmneg}, so we only prove Theorem \ref{mainthmcurv}.
\begin{proof}[Proof of Theorem \ref{mainthmcurv}]
	The main argument in Section \ref{MG} can still work if we replace the exponent $\delta(k,p)$ by $\tilde\delta(k,p)=\frac{n-1}3-\frac{2n-3}{3p}$. We just need to refine Case 2 to remove the log loss. We follow the strategy in Section \ref{refine}. We first consider $p\ge p_c$. By interpolation between  the $L^2-L^2$ bound  and the  $L^1-L^\infty$ bound we have
	\begin{equation}\label{intp12}
		\|T^\la\|_{L^{p'}(M)\to L^p(\Sigma)}\ls \la^{-\frac1p(n-\frac43)}.\end{equation}
	By interpolation between \eqref{intp12} and \eqref{intp2} we get
	\begin{equation}\label{bd12}
		\|T^\la\|_{L^{p_0'}(M)\to L^{p}(\Sigma)}\ls \la^{-\frac{n-2}{6p}-\frac{5n-6}{6p_0}}.
	\end{equation}
	Thus, inserting the power $\alpha$ from \eqref{bd12} into \eqref{prange}, we can obtain the lower bound
	\[p>\frac{2 (-4 n + n^2 + 5 n^3)}{24 + 4 n - 5 n^2 + 5 n^3}=2+\frac{12}{5n}-\frac4{5n^2}+O(\frac1{n^3})\]
	for $n\ge11$, and $p\ge p_c$ for $n\le 10$. It is easy to see that this bound is less than 2.3.
	
	Next, we can still consider $2\le p<p_c$ for $n\le 10$. To remove the log loss, we require
	\begin{equation}\label{cond12}n-2-\frac{n-1}p-\frac n{p_0}+\sigma(q_0)\le \frac{n-1}3-\frac {2n-3}{3p}
	\end{equation}
	\begin{equation}\label{cond22}
		n-2-\frac{n-1}p-\frac n{p_0}+(-\alpha-\frac{n-1}2+\frac n{p_0}+\frac {n-1}p)+\sigma(q_0)<\frac{n-1}3-\frac {2n-3}{3p}.
	\end{equation}
	To get the largest range of $p$, we require that $q_0=\frac{2n+2}{n-1}$ and $\frac1{p_0}=\frac{n+3}{2n+2}-\frac2n$ when $n\ge3$ and that $p_0=q_0=\infty$ when $n=2$. Then \eqref{cond12} always holds, and \eqref{cond22} is equivalent to 
	\begin{equation}\label{prange12}
		\alpha>\frac{2n-3}{3p}+\frac{n^2-3n-10}{6(n+1)},\ \ \text{for}\ n\ge3,
	\end{equation}
	and 
	\begin{equation}\label{prange122}
		\alpha>\frac1{3p}-\frac13,\ \ \text{for}\ n=2.
	\end{equation}
	Now it suffices to determine the best $\alpha$ in \eqref{Lpkey}. Indeed, 
	Thus, inserting the power $\alpha$ from \eqref{bd3} (when $n\le 5$)  or \eqref{bd12} (when $n\ge6$) into \eqref{prange12} or \eqref{prange122} we get 
	\[p>\begin{cases}
		1,\quad\quad\quad\quad\quad\quad\quad\quad n=2\\
		12/7\approx 1.714,\ \ \ \ \ \ \ \ n=3\\
		25/12\approx2.083,\ \ \ \ \ \ \ n=4\\
		35/16\approx2.188,\ \ \ \ \ \  \ n=5\\
		49/22\approx2.227,\ \ \ \  \ \ \ n=6\\
		56/25=2.240,\ \ \ \ \ \ \ n=7\\
		360/161\approx 2.236,\ \ \ \ \   n=8\\
		69/31\approx 2.226,\ \ \ \  \ \ \ \ \ n=9\\
		715/323\approx 2.214,\ \ \ \  \ \ n=10.
	\end{cases}\]
 So in particular these ranges cover the  endpoint $p=2$ when $n=2,3$. 
\end{proof}

\section{Restriction of toral eigenfunctions}
Let $\eps=\la^{-1}$ and  $\1_\la=\1_{[\la-\eps,\la+\eps]}$.  Note that the interval $[\la-\eps,\la+\eps]$ essentially contain at most  one eigenvalue of $P^0=\sqrt{-\Delta_g}$ on $\mathbb{T}^2$. 
By Bourgain-Rudnick \cite[Main Theorem]{br2012} and Huang-Zhang \cite[Theorem 2]{hzapde}, if $\Sigma\subset \mathbb{T}^2$ is a  curve segment with nonvanishing geodeisc curvature or a  segment of a closed geodesic, then for all $\la$ we have  \begin{equation}\label{tori1}
	\|\1_\la(P^0)f\|_{L^2(\Sigma)}\ls \|f\|_{L^2(\mathbb{T}^2)}.
\end{equation}   
Moreover, by Huang-Zhang \cite[Theorem 1]{hzapde}, if $\Sigma$ is a geodesic segment, then for all $\la$ we have
\begin{equation}\label{tori2}
	\|\1_\la(P^0)f\|_{L^2(\Sigma)}\ls\max_{r^2\in \mathbb{N}:|r-\la|\le\la^{-1}}\sqrt{N_{1,r}} \|f\|_{L^2(\mathbb{T}^2)},
\end{equation}  
where $N_{1,\la}$ is the maximum number of lattice points on an arc of length $\la^{1/2}$ on the circle $\{x\in \mathbb{R}^2:|x|=\la\}$. Bourgain-Rudnick \cite[Lemma 2.1]{br2015} proved that $N_{1,\la}\ls \log\la$. It was conjectured that $N_{1,\la}\ls 1$.

We shall extend these results for $H_V$ with $V\in L^2(\mathbb{T}^2)$. Our proof of Theorem \ref{BRthm} is a combination of the main argument in Section \ref{MG} with the following two well-known results. The first one is the uniform $L^4$ bounds by Cooke \cite{cook} and Zygmund \cite{zyg}: for all $\la$ we have
\begin{equation}\label{zyg}
	\|\1_\la(P^0)f\|_{L^4(\mathbb{T}^2)}\ls \|f\|_{L^2(\mathbb{T}^2)}.
\end{equation}
The second one is the spectral projection bounds by Bourgain-Burq-Zworski \cite[Prop. 2.4]{BBZ}: for all $\delta>\la^{-1}$ we have
\begin{equation}\label{proj}
	\|\1_{[\la,\la+\delta]}(P^0)\|_{L^4(\mathbb{T}^2)}\ls (\la\delta)^{\frac14}\|f\|_{L^2(\mathbb{T}^2)}.
\end{equation}
Using these bounds, Bourgain-Burq-Zworski \cite[Prop. 2.6]{BBZ} obtained  the uniform $L^4$ bounds  for $H_V$ with $V\in L^2(\mathbb{T}^2)$:
\begin{equation}\label{bbz}
	\|\1_\la(P_V)f\|_{L^4(\mathbb{T}^2)}\ls \|f\|_{L^2(\mathbb{T}^2)}.
\end{equation}

We just need to slightly modify the main argument in Section \ref{MG} to employ these estimates. First, we choose $p_0=q_0=4$ and $q=2$ throughout the proof. Second, we can employ the \eqref{proj} in Case 2 to remove the log loss there, since the bound in \eqref{proj} on the flat torus is smaller than the rough bound $(\la\delta)^{\frac12}$. Then running the main argument in Section \ref{MG} gives
\begin{equation}\label{final}
	\|\chi_\la(P_V)f\|_{L^2(\Sigma)}=\|\chi_\la(P^0)f\|_{L^2(\Sigma)}+O( \|V\|_{L^2}\|f\|_{L^2}).
\end{equation}
 So we complete the proof of Theorem \ref{BRthm} by using  \eqref{final} together with \eqref{tori1} and \eqref{tori2}.

\begin{remark}
	Recall that Bourgain-Rudnick \cite[Main Theorem]{br2012} also obtained uniform $L^2$-restriction estimates for real-analytic surfaces $\Sigma\subset \mathbb{T}^3$ with nonvanishing curvature.  It is an interesting open problem to extend  this result to $H_V$ with singular potentials. Due to the lack of the uniform $L^p$ bounds  on $\mathbb{T}^3$, currently we are not able to resolve this problem on $\mathbb{T}^3$.
\end{remark}
\section{Appendix: Further discussions on oscillatory integrals}\label{sectosc}

In the Section 4, we have removed the log loss for $k\le n/2$ when $n\ge5$ and for all $k$ when $n\le 4$. It remains to handle  $k>\frac{n}{2}$ when $n\ge5$. By Remark \ref{rem3.1}, we can reduce the problem to the estimate of the oscillatory integral operator in \eqref{oscT}. In this section, we shall further investigate this operator bound. We show that in the Euclidean model case, one can improve the range of $k$ to $k\le [2n/3]-2$ for $n\ge11$.

For the Euclidean distance, we have $d_g(x,y)=\sqrt{|u-y|^2+|x'|^2}$ with $x=(u,x')\in \mathbb{R}^k\times \mathbb{R}^{n-k}$ and $y\in \mathbb{R}^k$. Let
\begin{equation}\label{opt}
	T_\la f(x)=\int_{\mathbb{R}^k}e^{i\la\sqrt{|u-y|^2+|x'|^2}}a(x,y)f(y)dy,
\end{equation}
where $a\in C_0^\infty$ is supported in $\{(x,y): d_g(x,y)\approx 1\}$ and $|\partial_{x,y}^\gamma a(x,y)|\le C_\gamma.$  This is the adjoint operator of \eqref{oscT}. In our problem, we are  interested in the norm $$\|T_\la\|_{L^{p_c'}(\mathbb{R}^k)\to L^{p_0}(\mathbb{R}^n)}$$ where $\frac1{p_0}\in [\frac{n+3}{2n+2}-\frac2n,\frac{n-1}{2n+2}]$, $n\ge5$ and $k>n/2$. By Remark \ref{rem3.1}, we want to beat the bound $\la^{-\frac k{p_c}+\frac{n-1}2-\frac n{p_0}}$.

\subsection{Lower bound}

\noindent \textbf{Knapp type.}
Suppose that $y\sim 0$ and $x\sim e_n$ on the support of $a(x,y)$,  and $a(x,y)$ is real-valued and has a fixed sign on the support. Here $e_n=(0,...0,1)$. We fix $f(y)=1$ when $|y|<\frac1{100}\la^{-\frac12}$ and $f(y)=0$ otherwise. Let $$E=\{x=(u,x')\in\mathbb{R}^n:|u|<\frac1{100}\la^{-\frac12},\ |x-e_n|<\frac1{100}\}.$$  Then
\[|\sqrt{|u-y|^2+|x'|^2}-|x'||\le \frac1{10}\la^{-1}\]
if $x\in E$ and $|y|<\frac1{100}\la^{-\frac12}$. Thus,
\[|T_\la f(x)|=|e^{-i\la|x'|}T_\la f(x)|\approx \la^{-k/2},\ \ \forall x\in E.\]
We get
\[\|T_\la f\|_{L^{q}(\mathbb{R}^n)}\gs \la^{-\frac k2-\frac{k}{2q}},\]
and
\[\|f\|_{L^p(\mathbb{R}^k)}\approx \la^{-\frac k{2p}}.\]
So we have
\[\|T_\la\|_{L^p(\mathbb{R}^k)\to L^q(\mathbb{R}^n)}\gs \la^{-\frac k2(\frac1{p'}+\frac1q)}.\]
This lower bound is sharp when $p'=q$ and $k\le n-2$.

\noindent \textbf{Gaussian beam type.}
Suppose that $x\sim 0$ and $y\sim e_1$ on the support of $a(x,y)$,  and $a(x,y)$ is real-valued and has a fixed sign on the support. Here $e_1=(1,0,...,0)$. Let $y=(y_1,y')$. We fix $f(y)=e^{-i\la y_1}$ when $|y_1-1|<\frac1{100}$ and $|y'|<\frac1{100}\la^{-\frac12}$, and $f(y)=0$ otherwise. Let 
$$E=\{x=(u_1,u',x')\in \mathbb{R}^n:|u_1|<\frac1{100},\ |u'|^2+|x'|^2<\frac1{100}\la^{-1}\}$$
Then
$$|\sqrt{|y_1-u_1|^2+|u'-y'|^2+|x'|^2}-(y_1-u_1)|\le \frac1{10}\la^{-1}$$
if $x\in E$ and $|y'|<\frac1{100}\la^{-\frac12}$.
Thus,
$$|T_\la(x)|=|e^{i\la u_1}T_\la f(x)|\approx \la^{-\frac{k-1}2},\ \forall x\in E.$$
We get
\[\|T_\la f\|_{L^{q}(\mathbb{R}^n)}\gs \la^{-\frac{k-1}2-\frac{n-1}{2q}},\]
and
\[\|f\|_{L^p(\mathbb{R}^k)}\approx \la^{-\frac {k-1}{2p}}.\]
So we have
\[\|T_\la\|_{L^p(\mathbb{R}^k)\to L^q(\mathbb{R}^n)}\gs \la^{-\frac{k-1}{2p'}-\frac{n-1}{2q}}.\]
This lower bound is sharp when $p'=q$ and $k\ge n-2$.
\subsection{An upper bound} 
We shall use the Stein-Tomas argument to estimate an upper bound.  	We only consider $k\le n-2$ in the following, since  $p_c=2$ and we can apply $TT^*$ argument. 

For any fixed $x'=t\omega$ with $t=|x'|$ we analyze the operator 
\[T_\la^t f(u)=\int_{\mathbb{R}^k}e^{i\la \sqrt{|u-y|^2+t^2}}a(u,t\omega, y)f(y)dy.\]
We shall establish operator bounds for $T_\la^t$ independent of $\omega$. We first claim that if  $t,t'\approx r$ and $t\ne t'$ then
\begin{equation}\label{bound1}
	\|T_\la^tf\|_{L^2(\mathbb{R}^k)}\ls \la^{-k/2}r^{-(k+1)/2}\|f\|_{L^2(\mathbb{R}^k)},
\end{equation}
\begin{equation}\label{bound2}
		\|T_\la^t(T_\la^{t'})^*f\|_{L^\infty(\mathbb{R}^k)}\ls ((\la r^{-1}|t-t'|)^{-k/2}r^{-1}+(1+\la r^{11}|t-t'|)^{-N})\|f\|_{L^1(\mathbb{R}^k)},\ \forall N.
\end{equation}
Then by  the proof of the Stein-Tomas restriction theorem (see \cite[Corollary 0.3.7]{fio}), we have for $q=2(k+2)/k$
\begin{equation}\label{qbound1}(\int_{r}^{2r}\|T_\la^tf\|_{L^{q}(\mathbb{R}^k)}^qdt)^{\frac1q}\ls \la^{-\frac{k+1}q} r^{-\frac{k+12}{q}} \|f\|_{L^{2}(\mathbb{R}^k)}.\end{equation}
Moreover, by fixing one variable, one can verify the Carleson-Sj\"olin condition and apply Stein's oscillatory integral theorem (see \cite[Theorem 2.2.1]{fio}). We claim that
\begin{equation}\label{bound3}
	\|T_\la^tf\|_{L^2(\mathbb{R}^k)}\ls \la^{-(k-1)/2}\|f\|_{L^2(\mathbb{R}^k)}
\end{equation}
\begin{equation}\label{bound4}
	\|T_\la^tf\|_{L^{q_0}(\mathbb{R}^k)}\ls \la^{-k/q_0} \|f\|_{L^{2}(\mathbb{R}^k)}
\end{equation}
for $q_0=\frac{2(k+1)}{k-1}$, and then by interpolation we get
\begin{equation}
	\|T_\la^tf\|_{L^{q}(\mathbb{R}^k)}\ls \la^{-\frac{k+1}q}\la^{\frac{k+1}{kq}} \|f\|_{L^{2}(\mathbb{R}^k)}.
\end{equation}
So
\begin{equation}\label{qbound2}(\int_{r}^{2r}\|T_\la^tf\|_{L^{q}(\mathbb{R}^k)}^qdt)^{\frac1q}\ls \la^{-\frac{k+1}q}\la^{\frac{k+1}{kq}}r^{1/q} \|f\|_{L^{2}(\mathbb{R}^k)}.\end{equation}

Using the polar coordinate in $x'$ and the dyadic decomposition in  $t$, by \eqref{qbound1} and \eqref{qbound2} we get
\begin{align*}\|T_\la f\|_{L^q(\mathbb{R}^n)}&\ls \la^{-(k+1)/q}\Big(\sum_{j\ge0} \min\{2^{j(k+12)},\ \la^{\frac{k+1}k}2^{-j}\}\cdot 2^{-j(n-k-1)}\Big)^{1/q}\|f\|_{L^2(\mathbb{R}^k)}\\
	&=\la^{-\frac{(1 + k) (-13 + 11 k + k^2 + n)}{2 (2 + k) (13 + k)}}\|f\|_{L^2(\mathbb{R}^k)}.\end{align*}
 By the proof of \eqref{bgt} we have for $k\le n-3$
\[\|T_\la f\|_{L^2(\mathbb{R}^n)}\ls \la^{-k/2}\|f\|_{L^2(\mathbb{R}^k)}.\]
We  fix $\frac1{p_0}=\frac{n+3}{2n+2}-\frac2n$. Note that $p_0\in [2,q]$ when $k\le n-3$. So by interpolation we get
\begin{equation}
	\|T_\la f\|_{L^{p_0}(\mathbb{R}^n)}\ls \la^{\alpha(p_0,k,n)}\|f\|_{L^2(\mathbb{R}^k)},
\end{equation}
where \begin{align*}
	\alpha(p_0,k,n)&=\frac{26 + 11 n - n^2 + k (56 + 13 n - 14 n^2) + k^2 (6 + 2 n - n^2)}{2 (13 + k) n (1 + n)}\\
	&=-\frac{k}{2} + \frac{3k}{2n}-\frac{1}{2}+O(\frac{1}{n})
\end{align*}
for $n/2<k\le n-3$.
However, $$-\frac k2+\frac{n-1}2-\frac n{p_0}=-\frac k2+\frac12+O(\frac1n).$$ A direct calculation shows that  $\alpha(p_0,k,n)<-\frac k2+\frac{n-1}2-\frac n{p_0}$ 
 holds for $k\le [2n/3]-2$ when $n\ge11$. This improves the range  $k\le n/2$ obtained in Section 4.  

However, we have
\[\alpha(p_0,k,n)+\frac k2(\frac1{p_0}+\frac12)=\frac{(1 + k) (13 + 2 k - n) (2 + n)}{2 (13 + k) n (1 + n)}>0\]
for $n/2<k\le n-3$. So the upper bound here is strictly greater than the Knapp type lower bound. It would be interesting to find the sharp upper bound.

\begin{proof}[Proof of Claims]
	 Now we prove the claims \eqref{bound1}, \eqref{bound2}, \eqref{bound3}, \eqref{bound4}.

To prove \eqref{bound1}, by the $TT^*$ argument and Young's inequality, it suffices to estimate the kernel
\[K(y,z)=\int_{\mathbb{R}^k}e^{i\la(\sqrt{|u-y|^2+t^2}-\sqrt{|u-z|^2+t^2})}a(y,t\omega,u)\overline{a(u,t\omega,z)}du\]
and show that for $t\approx r$ we have
\begin{equation}\label{young11}
	\sup_y\int |K(y,z)|dz\ls \la^{-k}r^{-k-1}.
\end{equation}
Indeed,  we may assume that  $y=e_1$ and $z=z_1e_1$, where $e_1=(1,0,...,0)$ and $z_1\sim 1$. Let $u=(u_1,u')\in\mathbb{R}\times \mathbb{R}^{k-1}$. Then the phase function can be written as
\begin{equation}\label{phase}\phi(u)=\sqrt{|u_1-1|^2+|u'|^2+t^2}-\sqrt{|u_1-z_1|^2+|u'|^2+t^2}.\end{equation}
We get $|\partial_{u_1}\phi(u)|\approx  (r^2+|u'|^2)|y-z|$ and $|\partial_{u_1}^\alpha \phi(u)|\ls (r^2+|u'|^2)|y-z|$, $\forall \alpha$. Integration by parts in $u_1$ gives
\begin{align*}|K(y,z)|&\ls \int (1+\la(r^2+|u'|^2)|y-z|)^{-2N}d u'\\
	&\ls (1+\la r^2|y-z|)^{-N}(\la|y-z|)^{-\frac{k-1}2},\ \forall N.
\end{align*}
So we obtain \eqref{young11}.

To prove \eqref{bound2}, by Young's inequality we just need to estimate the kernel
\[K_{t,t'}(y,z)=\int_{\mathbb{R}^k}e^{i\la(\sqrt{|u-y|^2+t^2}-\sqrt{|u-z|^2+t'^2})}a(y,t\omega,u)\overline{a(u,t'\omega,z)}du\]
and show that if $t,t'\approx r$ and $t\ne t'$ then
\begin{equation}\label{young2}
	\sup_{y,z}|K_{t,t'}(y,z)|\ls (\la r^{-1}|t-t'|)^{-k/2}r^{-1}+(1+\la r^{11}|t-t'|)^{-N},\ \ \forall N.
\end{equation}
Indeed, as before may assume that  $y=e_1$ and $z=z_1e_1$ and the phase function
\begin{equation}\label{phase2}\phi(u)=\sqrt{|u_1-1|^2+|u'|^2+t^2}-\sqrt{|u_1-z_1|^2+|u'|^2+t'^2}.\end{equation}
Then $\partial_u\phi(u)=0$ has a unique zero $$ u_c=\frac{tz-t'y}{t-t'}=y+t(z-y)/(t-t')$$ and the Hessian matrix $$\partial_u^2\phi( u_c)=(\frac1t-\frac1{t'})(1+\beta^2)^{-\frac12}diag((1+\beta^2)^{-1},1,...,1)$$
where $\beta=|y-z|/|t-t'|$.  
Then we have either $|u-u_c|\ls1$ or  $|u-u_c|\approx r\beta$. This implies the upper bound
\begin{equation}\label{uppu}|(u-y)-\tfrac t{t'}(u-z)|=\frac{|t-t'|}{t'}|u-u_c|\ls r^{-1}(|t-t'|+r|y-z|).\end{equation}
When $|u-u_c|\gs1$, we also have the lower bound
\begin{equation}\label{lowu}|(u-y)-\tfrac t{t'}(u-z)|\gs r^{-1}(|t-t'|+r|y-z|).\end{equation}
Thus, when  $|u-u_c|\gs1$ we obtain 
\begin{align*}|\partial_u \phi(u)|&=\Big|\frac{u-y}{\sqrt{|u-y|^2+t^2}}-\frac{u-z}{\sqrt{|u-z|^2+t'^2}}\Big|\\
	&=\Big|\frac{u-y}{\sqrt{|u-y|^2+t^2}}-\frac{\frac t{t'}(u-z)}{\sqrt{|\frac t{t'}(u-z)|^2+t^2}}\Big|\\
	&\gs t^2|(u-y)-\tfrac t{t'}(u-z)|\\
&\approx r(|t-t'|+r|y-z|).
\end{align*}
Here we use the mean value theorem in the third line. Similarly, by \eqref{uppu} we get $$|\partial_u^\alpha \phi(u)|\ls r^{-1}(|t-t'|+r|y-z|),\ \ \forall  \alpha.$$ So when $|u-u_c|\gs1$, integration by parts gives the bound $(1+\la r^3|t-t'|)^{-N}$, which is better than the second bound in \eqref{young2}. So it suffices to consider $u\sim u_c$. This implies  $|t-t'|\gs r|y-z|$. In this case, we can similarly obtain $$|u-u_c|/|\partial_u\phi(u)|\ls r^{-1}|t-t'|^{-1}$$$$|\partial_u^\alpha \phi(u)|\ls r^{-1}|t-t'|,\ \ \forall \alpha.$$ By  stationary phase (see H\"ormander \cite[Theorem 7.7.5]{hor}) we get 
\[|K_{t,t'}(y,z)|\ls (\la r^{-1}|t-t'|)^{-k/2}r^{-1}+(1+\la r^{11}|t-t'|)^{-N},\ \ \forall N,\] which gives \eqref{young2}. The first term comes from the leading terms in the stationary phase expansion, and the second term comes from the remainder term.

To prove \eqref{bound3}  and \eqref{bound4}, we may assume that $x=(u,x')\sim 0$ and $y\sim e_1$ on the support of $a(x,y)$. Then $|u_1-y_1|\approx 1$. Fix $u_1,y_1$ and let $s^2=t^2+|u_1-y_1|^2$. Let
\[\tilde T_\la^{s}g(u')=\int_{\mathbb{R}^{k-1}}e^{i\la \sqrt{|u'-y'|^2+s^2}}a(u_1,u',t\omega, y_1,y')g(y')dy'.\]
Then \eqref{bound3} and \eqref{bound4} directly follows from the Minkovski inequality and the argument above, where $r$ is replaced by $s\approx 1$, and $k$ is replaced by $k-1$. 
\end{proof}

\subsubsection{Discussions of oscillatory integrals}

 Both of the lower bounds are still valid on general manifolds. The Knapp type lower bound is greater than the Gaussian beam type lower bound if and only if $k<n-1-\frac{q}{p'}$. Both of the lower bounds are sharp when $p'=q$. If we fix  $\frac1{p_0}=\frac{n+3}{2n+2}-\frac2n$, then  the lower bounds are strictly less than the bounds that we want to beat (see \eqref{prange}, \eqref{prange1}) when $(p,q)=(p_c',p_0)$. One might expect that these two examples together saturate the sharp upper bounds of the oscillatory integral operator.

The proof of the upper bound relies on the explicit formula of the Euclidean distance and the flatness of the submanifold. The main difficulty is that the rank of mixed Hessian $\partial_x\partial_y d_g(x,y)$ becomes degenerate (rank=$k-1$) when $x\in  \mathbb{R}^k\times \{0\}^{n-k}$. In general, this degeneracy happens when the geodesic connecting $x\in M$ and $y\in \Sigma$ is tangent to the submanifold $\Sigma$ at $y$. We call the point $x$ a degenerating point of $\Sigma$ if such degeneracy occurs for some $y\in\Sigma$. Let the degenerating set of $\Sigma$ be the collection of all such degenerating points. For example, in the above model case, the degenerating set of $\mathbb{R}^k$ is just $\mathbb{R}^k$ itself. However, in general, the dimension of the degenerating set can be as large as $\min\{2k,n\}$, making it difficult to precisely estimate the operator norm.

\section*{Declarations}
\noindent \textbf{Data availability statement.} Data sharing not applicable to this article as no datasets were generated or analyzed during the current study.

\noindent \textbf{Conflict of interests.} The authors have no relevant financial or non-financial interests to disclose.
	\bibliographystyle{plain}
	
\end{document}